\documentclass[12pt]{amsart}

\usepackage{enumerate}

\makeatletter

\@namedef{subjclassname@2010}{

  \textup{2010} Mathematics Subject Classification}

\usepackage{amssymb, amsmath,amsthm}
\newtheorem{thm}{Theorem}[section]
\newtheorem{prop}[thm]{Proposition}
\newtheorem{conj}[thm]{Conjecture}

\newtheorem{cor}[thm]{Corollary}
\newtheorem{lem}[thm]{Lemma}
\theoremstyle{definition}

\newtheorem{def1}[thm]{Definition}

\newcommand{\ra}{\rightarrow}
\newcommand{\bk}{\backslash}
\newcommand{\mc}{\mathcal}
\newcommand{\mf}{\mathfrak}
\newcommand{\mb}{\mathbb}
\newcommand{\sg}{\sigma}

\newcommand{\llf}{\left\lfloor}

\newcommand{\e}{\varepsilon}
\newcommand{\rrf}{\right\rfloor}

\newcommand{\mbf}{\boldsymbol}
\renewcommand{\bar}{\overline}

\begin{document}
\title{On the orbits of multiplicative pairs}
\author
{Oleksiy Klurman}
\address{Department of Mathematics, KTH Royal Institute of Technology, Stockholm, Sweden}
\email{lklurman@gmail.com}
\author{Alexander P. Mangerel}
\address{Centre de Recherches Math\'{e}matiques \\ Universit\'{e} de Montr\'{e}al\\
Montr\'{e}al, Qu\'{e}bec}
\email{smangerel@gmail.com}
\dedicatory{Dedicated to Imre K\'{a}tai on the occasion of his 80th birthday}

\begin{abstract}
We characterize all pairs of completely multiplicative functions $f,g:\mathbb{N}\to\mathbb{T}$, such that \[\overline{\{(f(n),g(n+1))\}_{n\ge 1}} \neq \mb{T} \times \mb{T}.\]
In so doing, we settle an old conjecture of Zolt\'{a}n Dar\'{o}czy and Imre K\'atai. \end{abstract}
\maketitle
\section{Introduction}
In this paper, we will be concerned with demonstrating yet another instance of the expected general phenomenon that the multiplicative structure of positive integers should in general be ``independent'' of their additive structure. Of principal focus here will be the behaviour of multiplicative functions at consecutive integers. \\
Problems of this kind are widely open in general, though spectacular progress has recently been made as a consequence of the breakthrough of Matom{\"a}ki and Radziwi{\l}{\l}~\cite{MaR}, and subsequent work of Matom{\"a}ki, Radziwi{\l}{\l} and Tao~\cite{MRT}, Tao \cite{Tao} and, more recently, Tao and Ter\"{a}v\"{a}inen~\cite{TaT}. In particular, using the work in \cite{MaR}, Tao~\cite{Tao} established a weighted version of the binary Chowla conjecture in the form
\[\sum_{n\le x}\frac{\lambda(n)\lambda(n+h)}{n}=o(\log x)\]
for all $h\ge 1.$ For a comprehensive account of the recent developments in this direction, see \cite{MRICM}.\\
Let $\mathbb{U}$ denote the unit disc in $\mathbb{C}$ and let $\mathbb{T}$ denote the unit circle. Let $f,g:\mathbb{N}\to\mathbb{T}$ be completely multiplicative functions. We expect that as $n$ varies through the set of positive integers, the values $f(n)$ and $g(n+1)$ should, roughly speaking, be independently distributed unless $f$ and $g$ satisfy some rigid relations. To be more precise, we shall investigate the following problem: if $\{f(n)\}_n$ and $\{g(n)\}_n$ are both dense in $\mb{T}$, but the sequence of pairs $\{(f(n),g(n+1))\}_n$ is \emph{not} dense in $\mb{T}^2$, must there be a rigid relation between $f$ and $g$? This multidimensional problem continues work on rigidity problems for additive and multiplicative functions initiated by the authors in \cite{RIG}. \\
This problem has a natural dynamical flavour, which explains the title of this paper. Let $T: \mb{N} \ra \mb{N}$ denote the rightward shift map $T(n) := n+1$. In this case, the above problem can be recast in terms of orbits of the pair $(f,g_{T})$, where $f,g: \mb{N} \ra \mb{T}$ are semigroup homomorphisms that fix $n = 1$, and $g_{T} := g \circ T$. We seek a result of the kind that, unless $f$ and $g$ are specially chosen maps, the orbit closure of the point $1$, i.e., $\bar{\{(f(n),g(n+1))\}_n}$, is expected to be the same as the product of the closures of the marginal orbits $\{f\circ T^n\}_n$ and $\{g\circ T^n\}_n$. \\
In this connection, we quote Conjecture 3 in the survey paper \cite{SUR} by K\'atai (earlier formulated in~\cite{DarKat}) which is the particular case $k=2$.
\begin{conj} \label{MULT}
Let $f,g: \mb{N} \ra \mb{T}$ be completely multiplicative. Suppose $\{(f(n),g(n+1))\}_n$ is not dense in $\mb{T}^2$, yet $\{f(n)\}_n$ and $\{g(n)\}_n$ are both dense in $\mb{T}$. Then there are integers $k$ and $l$ such that $f(n)^k = g(n)^l$, with $f(n) = n^{it}$ for some $t$.
\end{conj}
As stated this conjecture is easily seen to be false, as we can construct the following two types of counterexamples:
\begin{enumerate}[(i)]
\item Let $h_1,h_2:\mb{N}\to\mb{T}$ be completely multiplicative functions such that there are minimal positive integers $k,l \geq 2$ for which $h_1^k=h_2^{\l}=1.$ Fixing an arbitrary $t \in \mb{R}\bk \{0\}$ and setting $f(n) := h_1(n)n^{it}$ and $g(n) := h_2(n)n^{it}$ yields a pair of completely multiplicative functions such that $\{f(n)\}_n$ and $\{g(n)\}_n$ are dense, yet $\{(f(n),g(n+1))\}_n$ cannot be dense (as the components differ by at most a root of unity of bounded order). On the other hand, it is true in this example that $f(n)^k = n^{it'}$ and $g(n)^l = n^{it''}$ for all $n$, where $t' = kt$ and $t'' = lt$.
\item Fix a prime $p$ and distinct irrational numbers $\alpha, \beta \in \mb{R}$. Let $f,g:\mb{N} \ra \mb{T}$ be the completely multiplicative function defined on primes by 
$$f(q):= \begin{cases} e(\alpha) &: \text{ $q=p$}\\ 1 &: \text{ $q \neq p$}\end{cases} \ \ g(q):= \begin{cases} e(\beta) &: \text{ $q=p$}\\ 1 &: \text{ $q \neq p$.}\end{cases}$$
It is easy to see that the sequence $\{(f(n),g(n+1))\}_n$ belongs to the union of the sets $\mb{T} \times \{1\}$, $\{1\} \times \mb{T}$ and $\{1\} \times \{1\}$ and thus cannot be dense. On the other hand we clearly have $\overline{\{f(p^m)\}_{m\ge 1}}=\overline{\{g(q^m)\}_{m\ge 1}}=\mb{T}.$
\end{enumerate}
Given a completely multiplicative function $h$, let $$T_h := \{p \text{ prime}: h(p) \neq 1\}.$$ 
The collection (i) of counterexamples suggests that we should relax the conclusion of Conjecture \ref{MULT} by allowing $f(n)^k = n^{it}$ for some $k \geq 1$. The collection (ii) of counterexamples indicates that we should add a hypothesis to exclude those functions $f$ and $g$ such that both $|T_{f^k}||T_{g^k}| = 1$ and $T_{f^k} = T_{g^k}$ hold for all sufficiently large $k$. \\
We thus prove the following amendment of Conjecture \ref{MULT}, in which the above types of counterexamples are excluded.
\begin{thm}[Amended Dar\'{o}czy-K\'{a}tai Conjecture]\label{EKCFULL}
Let $f,g: \mb{N} \ra \mb{T}$ be completely multiplicative. Suppose $\overline{\{(f(n),g(n+1))\}_n}\ne \mb{T}^2,$ yet $\overline{\{f(n)\}_n}=\overline{\{g(n)\}_n}=\mb{T}.$ Suppose additionally that for infinitely many $m$, either $|T_{f^m}||T_{g^m}| > 1$ or $T_{f^m} \neq T_{g^m}$. Then there are integers $k$ and $l$ such that $f(n)^k = g(n)^l$, with $f(n)^{k} = n^{it}$ for some $t \in \mb{R}$. The above necessary conditions are also sufficient.
\end{thm}
In the case that $f = g$, we have the following immediate corollary.
\begin{cor}
Let $f: \mb{N} \ra \mb{T}$ be a completely multiplicative function such that $\overline{\{f(n)\}_n}=\overline{\{g(n)\}_n}=\mb{T}.$ Then $\overline{\{(f(n),f(n+1))\}_n}\ne \mb{T}^2$ if, and only if, one of the following conditions holds: \\
a) there is a positive integer $l$ such that $|T_{f^l}| = 1$, and for $p \in T_{f^l}$, $f(p)^l = e(\alpha)$ with $\alpha \notin \mb{Q}$; \\
b) there is a positive integer $k$ and a real number $t$ for which $f(n)^k = n^{it}$.
\end{cor}
In order to facilitate our discussion, we distinguish completely multiplicative functions according to their values on primes as follows. 
\begin{def1} \label{RATIRRAT}
A completely multiplicative function $f : \mb{N} \ra \mb{T}$ is said to be \emph{eventually rational} if there exist positive integers $k$ and $N_0 = N_0(k)$ such that for all $p \geq N_0$ we have $f(p)^k = 1$. We will say that $f$ is \emph{irrational} otherwise.
\end{def1}
An irrational function necessarily produces a sequence $\{f(n)\}_n$ that is dense. We stress, though, that the arguments of an irrational function at primes \emph{need not be} irrational. For example, by our definition, the completely multiplicative function defined by $f(p) = e(1/p)$ for all $p$ is irrational. \\ \\
To prove Theorem \ref{EKCFULL}, we will treat two cases, depending on whether or not $f$ or $g$ is irrational (in the sense of Definition \ref{RATIRRAT}). In Section \ref{SECEVRAT} we prove the following.\begin{prop} \label{EVRAT}
Suppose $f,g: \mb{N} \ra \mb{T}$ are eventually rational, $\overline{\{f(n)\}_n}=\overline{\{g(n)\}_n}=\mb{T}$ and either $|T_{f^k}||T_{g^k}| > 1$ or $T_{f^k} \neq T_{g^k}$ for infinitely many $k$. Then $\overline{\{(f(n),g(n+1))\}_n}=\mb{T}^2.$ 
\end{prop}
Proposition \ref{EVRAT} asserts that the only cases covered by Conjecture \ref{EKCFULL} are those for which either $f$ or $g$ is irrational. In this direction we prove the following in Section \ref{SECIRRAT}.
\begin{thm} \label{IRRAT}
Suppose that $f,g: \mb{N} \ra \mb{T}$ are completely multiplicative functions such that $\overline{\{f(n)\}_n}=\overline{\{g(n)\}_n}=\mb{T}.$ Suppose furthermore that at least one of $f$ and $g$ is irrational. If $\overline{\{(f(n),g(n+1))\}_n}\ne \mb{T}^2$ then there are positive integers $k,l$ and real numbers $t,t'$ such that $f(n)^k = n^{it}$ and $g(n)^l = n^{it'}$.
\end{thm}
Having shown Theorem \ref{IRRAT}, in order to prove Theorem \ref{EKCFULL} it remains to prove that the real numbers $t$ and $t'$ satisfy $t = \lambda t'$, for $\lambda \in \mb{Q}$. This is the conclusion
of Proposition \ref{RATRATIO}, which we prove in Section \ref{SECRATRATIO}.\\
The proof of Theorem \ref{IRRAT} uses arguments that extend the one used in the proof of Theorems 1.1 and 1.4 in \cite{RIG}. Let us recall the rough outline of this argument here. \\
In Theorem 1.4 of \cite{RIG} (the proof of which establishes Theorem 1.1 there as well), it was shown that the sequence $\{f(n)\bar{f(n+1)}\}_n$ is dense in $\mb{T}$, except in predictable cases\footnote{That is, except when $f(n) = g(n)n^{it}$, where $g$ is a function taking values in bounded order roots of unity.}. The objective was to show that for every $\e > 0$ and every $z \in \mb{T}$, the lower bound
\begin{equation}\label{GOODCOND}
|f(n)\bar{f(n+1)}-z| \geq \e \text{ for all $n$ sufficiently large}
\end{equation}
cannot hold for ``generic'' completely multiplicative functions $f: \mb{N} \ra \mb{T}$. To establish this, we noted that \eqref{GOODCOND} implies that the sequence $\{f(n)\bar{f(n+1)}\}_n$ cannot be equidistributed, which led us (via the Erd\H{o}s-Tur\'{a}n inequality and Tao's theorem on logarithmic averages of binary correlations \cite{Tao}; see Theorem \ref{TAOBINTHM} below) to a first conclusion that $f$ is \emph{pseudo-pretentious}, i.e., such that for some (minimal) $k \in \mb{N}$, some Dirichlet character $\chi$ and some $t \in \mb{R}$ (depending at most on $\e$), we have
$$\mb{D}(f^k,\chi n^{it};x)^2 := \sum_{p \leq x} \frac{1-\text{Re}(f(p)^k \bar{\chi}(p)p^{-it})}{p} \ll_{\e} 1;$$
in particular, $\mb{D}(f,gn^{it};x) \ll_{\e} 1$ for some completely multiplicative function $g$ such that $g(n)^k = \chi(n)$ whenever $\chi(n) \neq 0$. Here the fact that $(n+1)^{it} \approx n^{it}$ for large $n$ is used crucially. Roughly speaking, we then conditioned on a suitable subset of $n$ (or positive logarithmic density) such that $g(n)\bar{g(n+1)}$ is constant, and thereby reduced our work to treating the 1-pretentious function $F = f\bar{g}n^{-it}$, showing that for small enough $\e$,
$$|F(n)\bar{F(n+1)}-z'| = |f(n)\bar{f(n+1)}-z| + O(\e^2) \gg \e$$ for some $z'$ (where $z$ and $z'$ need not be the same) is untenable for all $n$ sufficiently large (depending at most on $\e$).  This reduction was a key part that made that argument work. \\
In the context of Theorem \ref{EKCFULL}, we must face several key differences in the argument. For example: \\
\begin{enumerate}[(i)]
\item the fact that $\|(f(n),g(n+1))-(z,w)\|_{\ell_1} \geq \e$ for all large $n$ may mean that $|f(n)-z| \geq \e/2$ on a very dense set, or a very sparse set, and we cannot exclude either of these possibilities;\\
\item if $f$ and $g$ are both pseudo-pretentious in the above sense, say $f$ is pretentious to $h_1n^{it}$ and $g$ is pretentious to $h_2n^{it'}$  then it may be that $t \neq t'$, and we must then deal with the distribution in argument of the twist $n^{i(t-t')}$, unlike in the outline of the proof of Theorem 1.4 of \cite{RIG}. 
\end{enumerate}

With some additional ideas, we are able to address these issues. See especially Section \ref{SECIRRAT} for further details.
\section*{Ackhowledgments}
The authors are grateful to Andrew Granville, John Friedlander  and Imre K\'atai for their interest in the results of the present paper.
The bulk of this work was completed while the second author was a Ph.D student at the University of Toronto. He would like to thank U of T for excellent working conditions while this project was being completed.
\section{Auxiliary Results towards Theorem \ref{EKCFULL}} \label{SECAUX}
 In this section we collect the definitions and lemmata that we shall use in the proof of Theorem \ref{EKCFULL}.\\
 A crucial result on which our method relies is the following recent breakthrough result of Tao \cite{Tao}.
 \begin{thm}[Tao] \label{TAOBINTHM}
 Let $f_1,f_2 : \mb{N} \ra \mb{U}$ be multiplicative functions, such that for some $j \in \{1,2\}$, we have 
$$\inf_{|t| \leq x} \mb{D}(f_j,\chi n^{it};x)^2 = \inf_{|t| \leq X} \sum_{p \leq X} \frac{1-\text{Re}(f_j(p)\bar{\chi}(p)p^{-it})}{p} \ra \infty$$ 
for each fixed Dirichlet character $\chi$. Then for any integers $a_1,a_2 \geq 1$ and $b_1,b_2 \geq 0$ with $a_1b_2 - a_2b_1 \neq 0$, we have
 \begin{equation*}
\frac{1}{\log x} \sum_{n \leq x} \frac{f_1(a_1n+b_1)f_2(a_2n+b_2)}{n} = o(1).
 \end{equation*}
 \end{thm}
A useful consequence of Theorem \ref{TAOBINTHM} is the following.
\begin{prop} \label{ULTPRETLEM}
Let $f,g : \mb{N} \ra \mb{U}$ be multiplicative functions. Suppose $k,l \in \mb{N}$ are minimal, such that
\begin{equation}\label{LARGEBINCORHYP}
\left|\frac{1}{\log x} \sum_{n \leq x} \frac{f(n)^k g(n+1)^l}{n}\right| \gg 1,
\end{equation}
as $x \ra \infty$. Then there are Dirichlet characters $\chi_1,\chi_2$ with respective conductors $q_1,q_2 = O(1)$, and real numbers $t_1,t_2 = O(1)$ such that $\mb{D}(f,h_1n^{it_1},\infty)<\infty$ and  $\mb{D}(f,h_2n^{it_2},\infty)<\infty$, with\footnote{Given $a,b \in \mb{N}$, we write $(a,b^{\infty}) := \prod_{p|b}p^{\nu_p(a)}$.} $h_1(n)^k = \chi_1(n/(n,q_1^{\infty}))$ and $h_2(n)^l = \chi_2(n/(n,q_2^{\infty}))$.
\end{prop}
\begin{proof}
From \eqref{LARGEBINCORHYP} and Theorem \ref{TAOBINTHM}, it follows that for each $x$ sufficiently large there is a pair $(\xi_x,t_x)$, where $\xi_x$ is a primitive Dirichlet character with $\text{cond}(\xi_x) \ll 1$ and $t_x \ll 1$, each estimate being uniform in $x$, such that $\mb{D}(f^k,\xi_xn^{it_x};x) \ll 1$. By Lemma 2.5 of \cite{RIG}, it follows that there is a character $\chi_1$ and a $t_1 \in \mb{R}$ such that $\mb{D}(f^k,\chi_1n^{ikt_1};x) \ll 1$. Combining this with Proposition 2.8 of \cite{RIG}, it follows that there is a completely multiplicative function $h_1$ such that for all primes $p \nmid \text{cond}(\chi_1)$ we have $h_1(p)^k = \chi_1(p)$, while if $p|\text{cond}(\chi_1)$, $h_1(p) = 1$. This implies the claim about $f$. The claim about $g$ follows in the same way.
\end{proof}
\subsection{Presieving on Level Sets With Archimedean Twists}
Recall that a 1-bounded multiplicative function $f$ is called \emph{pseudo-pretentious}
if there exists a (minimal) positive integer $k$, a primitive Dirichlet character $\chi$ modulo $q$, a real number $t$ and a completely multiplicative function $h : \mb{N} \ra \mb{T}$ such that $h(n)^k = \tilde{\chi}(n)$, and $\mb{D}(f,h n^{it};x) \ll 1$. Here, $\tilde{\chi}$ is the unimodular completely multiplicative function defined on primes via $\tilde{\chi}(p) = \chi(p)$ if $p \nmid q$, and $\tilde{\chi}(p) = 1$ otherwise. Henceforth, we will refer to the function $h$ here as a \emph{pseudocharacter}, and refer to the modulus $q$ of $\chi$ as the \emph{conductor} of $h$.
We emphasize that in this definition the minimality of $k$ implies that $h(n)^j$ is non-pretentious for all $1 \leq j \leq k-1.$ This will be crucial in several of the arguments below. \\
Given $\alpha_j \in h_j^{-1}(\mb{N})$ for $j = 1,2$ and $N,B \geq 1$, write
\begin{equation*}
\mc{A}_{N,B}(h_1,h_2;\alpha_1,\alpha_2) := \{n \in \mb{N} : P^-(n(Bn+1)) > N, h_1(n) = \alpha_1, h_2(Bn+1) = \alpha_2\}.
\end{equation*}
Furthermore, if $I,J \subseteq \mb{T}$ are arcs\footnote{By an \emph{arc} in $\mb{T}$, we mean the image of an interval $[a,b] \subseteq \mb{R}$ under the exponentiation map $t \mapsto e(t)$. Thus, a symmetric arc about 1, for example, refers to the image under exponentiation of any interval $[m-\eta,m+\eta]$ with $m \in \mb{Z}$.} and $u,v \in \mb{R}$ then
\begin{align*}
&\mc{A}_{N,B,I}(h_1,h_2;\alpha_1,\alpha_2;u) := \{n \in \mc{A}_{N,B}(h_1,h_2;\alpha_1,\alpha_2) : n^{iu} \in I\} \\
&\mc{A}_{N,B,I,J}(h_1,h_2;\alpha_1,\alpha_2;u,v) := \{n \in \mc{A}_{N,B}(h_1,h_2;\alpha_1,\alpha_2) : n^{iu} \in I, n^{iv} \in J\}.
\end{align*}
Moreover, if $x,q \geq 1$ and $a$ is a coprime residue class modulo $q$, put
$$\Phi_{N,B}(x;q,a) := |\{n \leq x : P^-(n(Bn+1)) > N \text{ and } n \equiv a(q)\}|.$$
When $q =1$, we write $\Phi_{N,B}(x;q,a) = \Phi_{N,B}(x)$, and when $B = 1$ in addition, we write $\Phi_N(x)$. Finally, given a map $f: \mb{N} \ra \mb{C}$ and a positive real $X \geq 2$, we shall write
\begin{align*}
\mb{E}_{n \leq X}^{\log} f(n) &:= \frac{1}{\log X} \sum_{n \leq X} \frac{f(n)}{n}.
\end{align*}
Our first result, to be used throughout the paper, shows that the sets $\mc{A}_{N,B,I}$ and $\mc{A}_{N,B,I,J}$ given above, on which the values of two discrete functions $h_1,h_2$ as well as archimedean characters $n^{iu}$ and $n^{iv}$, can be large in the sense that they have positive upper density\footnote{To be precise, Proposition \ref{ONES} implies directly that these sets have positive upper \emph{logarithmic density}, i.e., $\limsup_{x \ra \infty} \mb{E}_{n \leq x}^{\log} 1_{\mc{A}}(n) > 0$; however, this also implies that the upper density $\limsup_{x \ra \infty} \frac{1}{x}\sum_{n \leq x} 1_{\mc{A}}(n) > 0$ as well.} in general.
\begin{prop} \label{ONES}
Let $\chi_1,\chi_2$ be primitive Dirichlet characters of respective conductors $q_1$ and $q_2$ and orders $k$ and $l$. Let $h_1,h_2 : \mb{N} \ra \mb{T}$ be pseudocharacters modulo $q_1$ and $q_2$, respectively, such that $h_1(n)^r = \tilde{\chi_1}(n)$ and $h_2(n)^s = \tilde{\chi_2}(n)$, for some $r,s \in \mb{N}$. \\
a) Let $u \in \mb{R}$, and let $(\alpha,\beta) \in \mu_{kr}\times\mu_{ls}$. 
Let $\delta > 0$, $z \in \mb{T}$ and let $I \subset \mb{T}$ be an arc with length $\delta$ about $1$. 
Then, for any $B \geq 1$ satisfying $2q_1q_2|B$, 
$$\mb{E}_{n \leq x}^{\log} 1_{\mc{A}_{N,B,I}(h_1,h_2;\alpha,\beta;u)}(n) \gg \frac{\delta}{krls}\frac{\Phi_{N,B}(x)}{x}$$ 
as $x \ra \infty$. Moreover, if $u \neq 0$ then we may replace $I$ above by any arc of length $\delta$.\\
b) If $u,v \in \mb{R}$ are fixed and such that $u/v \notin \mb{Q}$, and $J_1$ and $J_2$ are arcs in $\mb{T}$ of respective lengths $\delta_1$ and $\delta_2$, then 
$$\mb{E}_{n \leq x}^{\log}1_{\mc{A}_{N,B,J_1,J_2}(h_1,h_2;\alpha,\beta;u,v)} \gg \frac{\delta_1\delta_2}{krls} \frac{\Phi_{N,B}(x)}{x}.$$
\end{prop}
To prove Proposition \ref{ONES}, we need the following variant of Lemma 2.11 from \cite{RIG}, which is easily proven by the Fundamental lemma of the sieve. 
\begin{lem} \label{TRIV}
Let $q,B \geq 1$, $N \geq 2$ and let $a$ be a residue class modulo $q$ such that $(a(Ba+1),q) = 1$. Then as $x \ra \infty$,
\begin{equation*}
\Phi_{N,B}(x;q,a) = \frac{x}{q}\prod_{p|B/(B,q)}\left(1-\frac{1}{p}\right)\prod_{3 \leq p \leq N \atop p\nmid q} \left(1-\frac{2}{p}\right) + O(4^{\pi(N)}).
\end{equation*}
\end{lem}
In the sequel, we write
$$\delta_{N,B,q} := \frac{1}{q}\prod_{p|B/(B,q)}\left(1-\frac{1}{p}\right)\prod_{3 \leq p \leq N \atop p\nmid q} \left(1-\frac{2}{p}\right),$$ and set $\delta_{N,B} = \delta_{N,B,1}$.
\begin{lem} \label{TWIST}
Assume the hypotheses of Proposition \ref{ONES}. Furthermore, suppose $N > \max\{q_1,q_2\}$, and that $h_1^j$ and $h_2^m$ are both non-pretentious for all $1 \leq j \leq r-1$ and $1 \leq m \leq s-1$. Then as $x \ra \infty$,
\begin{equation*}
\sum_{n \leq x} \frac{1_{A_{N,B}(h_1,h_2;\alpha,\beta)}(n)}{n^{1+iu}} = \begin{cases} \frac{\delta_{N,B}}{iukrls} + o(\log x) + O_{q_1}\left(|u| 4^{\pi(N)}\right) &\text{ if $u \neq 0$} \\ \frac{1}{krls}\left(\delta_{N,B}+ o(1)\right) \log x + O_{q_1}\left(4^{\pi(N)}\right) &\text{ otherwise.}\end{cases}
\end{equation*}
\end{lem}
\begin{proof}
Since $\alpha \in \mu_{kr}$ and $\beta \in \mu_{ls}$, we have the identities
\begin{align*}
1_{h_1(n) = \alpha} &= \frac{1}{kr} \sum_{0 \leq j \leq kr-1} (h_1(n)\bar{\alpha})^j\\
1_{h_2(n) = \beta} &= \frac{1}{ls} \sum_{0 \leq m \leq ls-1} (h_2(n)\bar{\beta})^j.
\end{align*}
It follows that
\begin{align}
&\sum_{n \leq x} \frac{1_{\mc{A}_{N,B}(h_1,h_2;\alpha,\beta)}(n)}{n^{1+iu}} = \sum_{n \leq x \atop P^-(n(Bn+1)) > N} \frac{1_{h_1(n) = \alpha}1_{h_2(n) = \beta}}{n^{1+iu}} \nonumber\\
&= \frac{1}{klrs}\sum_{0 \leq j \leq kr-1} \sum_{0 \leq m \leq ls-1} \alpha^{-j}\beta^{-m} \sum_{n \leq x \atop P^-(n(Bn+1)) > N} \frac{h_1(n)^jh_2(n)^mn^{-iu}}{n}. \label{EXPWITHLAG}
\end{align}
As $h_1^u$ and $h_2^v$ are both non-pretentious for all $1 \leq u \leq r-1$ and $1 \leq v \leq s-1$, it follows that the multiplicative functions
\begin{align*}
\phi_{u+ar}(n) &:= h_1(n)^{j+ar}1_{P^-(n)>N}n^{-iu} = h_1(n)^j\tilde{\chi}_1(n)^a1_{P^-(n)>N}n^{-iu}\\
\psi_{v+bs}(n) &:= h_2(n)^{m+bs}1_{P^-(n) > N} = h_2(n)^m\tilde{\chi}_2(n)^b1_{P^-(n) > N}
\end{align*}
are both non-pretentious for such $u,v$, and any $0 \leq a \leq k-1$, $0 \leq b \leq s-1$.
By Theorem \ref{TAOBINTHM}, it follows that
\begin{align}
&\left|\mathop{\sum_{0 \leq j \leq kr-1 } \sum_{0 \leq m \leq ls-1}}_{r\nmid j \text{ or } s \nmid m}  \alpha^{-j}\beta^{-m} \sum_{n \leq x \atop P^-(n(Bn+1)) > N} \frac{h_1(n)^jh_2(n)^mn^{-iu}}{n} \right| \nonumber\\
&\leq \mathop{\sum_{0 \leq j \leq kr-1 } \sum_{0 \leq m \leq ls-1}}_{r\nmid j \text{ or } s \nmid m}\left|\sum_{n \leq x} \frac{\phi_j(n)\psi_m(Bn+1)}{n}\right| = o(\log x), \label{EASYONES}
\end{align}
where the estimate depends on $k,r,l$ and $s$.
%
When $j = ra$ and $m = bs$, we instead have
\begin{equation}
\sum_{n \leq x} \frac{\phi_{ar}(n)\psi_{bs}(Bn+1)}{n} = \sum_{n \leq x \atop P^-(n(Bn+1)) > N} \frac{\tilde{\chi}_1(n)^{a} \tilde{\chi}_2(Bn+1)^{b}}{n^{1+iu}} = \sum_{n \leq x \atop P^-(n(Bn+1)) > N} \frac{\tilde{\chi}_1(n)^{a}}{n^{1+iu}} \label{NOCORR}
\end{equation}
for each $b$, as $q_2|B$. Now, as $N > q_1$, for each $n$ with $P^-(n(Bn+1)) > B$ we must have $(n,q_1) = 1$. Thus, splitting the rightmost sum in \eqref{NOCORR} into coprime residue classes modulo $q_1$, the RHS of \eqref{NOCORR} becomes
\begin{equation*}
\sum_{c(q_1)}\chi_1(c)^{a} \sum_{n \leq x \atop P^-(n(Bn+1)) > N} \frac{1_{n \equiv c (q_1)}}{n^{1+iu}}.
\end{equation*}
We first consider the case $u \neq 0$. Applying the previous lemma and partial summation, we get that for each coprime residue $c$ modulo $q_1$,
\begin{align}
\sum_{n \leq x \atop P^-(n(Bn+1)) > N} \frac{1_{n \equiv c(q_1)}}{n^{1+iu}} &= \int_1^x \frac{1}{t^{1+iu}} d\left\{\Phi_{N,B}(x;q_1,c)\right\} \nonumber\\
&= \delta_{N,B,q_1}\int_1^x \frac{dt}{t^{1+iu}} + O\left((q_1+|u|4^{\pi(N)}\int_1^x \frac{dt}{t^2}\right) \label{U1}\\
&= \frac{\delta_{N,B,q_1}}{iu}\left(1-x^{-iu}\right) + O(q_1(1+|u|)4^{\pi(N)}) \label{U2EQ}.
\end{align}
The main term being independent of $c$ modulo $q_1$, we can invert the orders of summation and use orthogonality to get that whenever $a \neq 0$, we have
\begin{align*}
\sum_{c(q_1)}\chi_1(c)^{a}\sum_{n \leq x \atop P^-(n(Bn+1)) > N} \frac{1_{n \equiv c (q_1)}}{n^{1+iu}} &= \frac{\delta_{N,B,q_1}}{iu}\left(1-x^{-iu}\right)\sum_{c(q_1)} \chi_1(c)^a + O\left(q_1^2(1+|u|4^{\pi(N)})\right) \\
&\ll q_1^2\left((1+|u|)4^{\pi(N)}\right).
\end{align*}
As such, it follows that whenever $a \neq 0$, we have
\begin{equation}\label{BYPERIOD}
\sum_{n \leq x} \frac{\phi_{ar}(n)\psi_{bs}(Bn+1)}{n} \ll q_1^2\left((1+|u|)4^{\pi(N)}\right).
\end{equation}
In the remaining case $a = 0$, we have
\begin{equation}\label{MAINT}
\sum_{n \leq x \atop P^-(n(Bn+1)) > N} \frac{1}{n^{1+iu}} = \frac{\delta_{N,B}}{iu}\left(1-x^{-iu}\right) + O\left((1+|u|)4^{\pi(N)}\right),
\end{equation}
as this is the the LHS of \eqref{U1}, but with $q_1$ replaced by $1$. \\
Combining \eqref{EASYONES}, \eqref{BYPERIOD} and \eqref{MAINT}, and inserting these into \eqref{EXPWITHLAG}, we find that
$$\sum_{n \leq x} \frac{1_{\mc{A}_{N,B}(h_1,h_2;\alpha,\beta)}(n)}{n^{1+iu}} = \frac{1}{krls}\left(\frac{\delta_{N,B}}{iu}\left(1-x^{-iu}\right) + O\left(q_1(1+|u|)4^{\pi(N)}\right) + o(\log x)\right)$$
when $u \neq 0$, as claimed.\\
The case $u = 0$ follows the same lines, but is simpler. 
The proof is now complete.
\end{proof}
\noindent Before proceeding to the proof of Proposition \ref{ONES}, we pause to recall the following objects of relevance to it. \\
Given $K \in \mb{N}$, let $B_K: \mb{R}/\mb{Z} \ra \mb{C}$ denote the degree $K$ \emph{Beurling polynomial} (see Section 1.2 of \cite{Mon} for a definition and some of its properties). Recall that it satisfies the properties that: \\
i) if $\psi(t) := \{t\}-1/2$ is the sawtooth function then $-B_K(-t) \leq \psi(t) \leq B_K(t)$. As such, given an interval $I \subseteq \mb{R}/\mb{Z}$ with endpoints $0 \leq a < b < 1$, we have
\begin{equation} \label{BKUPPLOW}
|I| - B_K(b - t) - B_K(t-a) \leq 1_I(t) \leq |I| +B_K(t-b) + B_K(a- t).
\end{equation}
ii) the Fourier coefficients\footnote{As usual, for $t \in \mb{R}$ we write $e(t) := e^{2\pi i t}$.} $\hat{B}_K(m) := \int_0^1 B_K(t)e(-mt) dt$ satisfy 
\begin{equation*}
\hat{B}_K(m) = \begin{cases} 0 &\text{ for $|m| > K$} \\ \frac{1}{2(K+1)} &\text{ if $m = 0$} \\ O\left(\frac{1}{m}\right) &\text{ for all $m \neq 0$,}\end{cases}
\end{equation*}
the implicit constant in the last estimate being absolute. 
\begin{proof}[Proof of Proposition \ref{ONES}]
a) When $u = 0$ it is clear that 
$$\mc{A}_{N,B,I}(h_1,h_2;\alpha,\beta;u) = \{n \in \mc{A}_{N,B}(h_1,h_2;\alpha,\beta) : 1 \in I\} = \mc{A}_{N,B}(h_1,h_2;\alpha,\beta),$$ since $1 \in I$. 
By definition, we get
\begin{equation}\label{DELTNBBD}
\delta_{N,B} = \frac{\Phi_{N,B}(x)}{x} + O\left(4^{\pi(N)}x^{-1}\right).
\end{equation}
Proposition \ref{ONES} in the case $u = 0$ thus follows from Lemma \ref{TWIST}, since
$$\sum_{n \leq x}  \frac{1_{\mc{A}_{N,B}(h_1,h_2;\alpha,\beta)}(n)}{n} = \left(\frac{\delta_{N,B,1}}{krls}+o(1)\right) \log x + O\left(4^{\pi(N)}\right) \gg \frac{\delta}{krls}\frac{\Phi_{N,B}(x)}{x} \log x.$$
We now assume that $u \neq 0$. Here, we no longer assume that $I$ is an arc around $1$, but rather any arc of length $\delta$ (the proof being the same regardless of the centre point of the arc). Let $K$ be a large integer, and suppose $x$ is sufficiently large in terms of $K.$ Write $I' \subseteq \mb{R}/\mb{Z}$ to be the interval that maps onto $I$ under exponentiation. From \eqref{BKUPPLOW}, it follows that 
\begin{align}
&\sum_{n\leq x} \frac{1_{\mc{A}_{N,B,I}(h_1,h_2;\alpha,\beta;u)}(n)}{n} = \sum_{n \in \mc{A}_{N,B}(h_1,h_2;\alpha,\beta)} \frac{1_{I'}\left(\left\{\frac{u\log n}{2\pi}\right\}\right)}{n} \nonumber\\
&\geq |I'|\sum_{n \leq x} \frac{1_{\mc{A}_{N,B}(h_1,h_2;\alpha,\beta)}(n)}{n} - \sum_{|m| \leq K}c_m\left(e(mb) \sum_{n \leq x}\frac{1_{\mc{A}_{N,B}(h_1,h_2;\alpha,\beta)}(n)}{n^{1+imu}}+ e(-ma) \sum_{n\leq x}\frac{1_{\mc{A}_{N,B}(h_1,h_2;\alpha,\beta)}(n)}{n^{1-imu}}\right) \nonumber\\
&\geq \frac{\delta \delta_{N,B}}{krls} \log x - 2\sum_{|m| \leq K}|\hat{B}_K(m)|\left|\sum_{n \leq x}\frac{1_{\mc{A}_{N,B}(h_1,h_2;\alpha,\beta)}(n)}{n^{1+imu}}\right|, \label{MINOR}
\end{align}
where we used $|I'| \asymp |I| =\delta$ in the last line. We consider the second term on the RHS of \eqref{MINOR}. Applying Lemma \ref{TWIST} for each $m$ and summing, we get
\begin{align*}
\sum_{1 \leq |m| \leq K} |\hat{B}_K(m)|\left|\sum_{n \leq x}  \frac{1_{\mc{A}_{N,B}(h_1,h_2;\alpha,\beta)}(n)}{n^{1+imu}}\right| &\ll \sum_{1 \leq m \leq K} \frac{1}{m} \left(\frac{\delta_{N,B}}{m|u|} + o(\log x) + (1+m|u|)4^{\pi(N)}\right) \\
&\ll |u|^{-1} + o((\log K)(\log x)) + K^2|u|4^{\pi(N)}.
\end{align*}  
The term with $k = 0$ gives
$$|\hat{B}_K(0)| \sum_{n \in \mc{A}_{N,B}(h_1,h_2;\alpha,\beta)} \frac{1}{n} \leq \frac{\log x}{2K}.$$
As such, we get
\begin{equation*}
\sum_{n \in \mc{A}_{N,B,I}(h_1,h_2;u)} \frac{1}{n} \geq \left(\frac{\delta \delta_{N,B}}{krls}-\frac{1}{2K} - o(\log K)\right) \log x - O\left(|u|^{-1} + K^2|u|4^{\pi(N)}\right).
\end{equation*}
Choosing $K = K(x)$ tending to infinity with $x$, but sufficiently slow-growing so that $o((\log K)(\log x)) = o(\log x)$ and $K^24^{\pi(N)} = o(\log x)$, proves the claim.\\
b) The case when $uv = 0$ is similar to a) and we assume $uv \neq 0$. We are interested in 
\begin{equation} \label{WITHTWOCONDS}
\sum_{n \leq x \atop n \in \mc{A}_{N,B}(h_1,h_2;\alpha,\beta)} \frac{1_{J_{1}}(n^{iu})1_{J_{2}}(n^{iv})}{n}.
\end{equation}
We minorize each indicator function by a Beurling polynomial of degree $K$ in the same way as was done just above, and we merely sketch the argument, showing where the condition $u/v \notin \mb{Q}$ is relevant. By the triangle inequality, \eqref{WITHTWOCONDS} is bounded from below by
\begin{align*}
&|J_{1}||J_{2}|\sum_{n \leq x \atop n \in \mc{A}_{N,B}(h_1,h_2;\alpha,\beta)}\frac{1}{n}\\
&-\left(\sum_{1 \leq |m_1| \leq K}\frac{1}{|m_1|}\left|\sum_{n \in \mc{A}_{N,B}(h_1,h_2;\alpha,\beta)} \frac{1}{n^{1+im_1u}}\right| + \sum_{1 \leq |m_2| \leq K} \frac{1}{|m_2|}\left|\sum_{n \in \mc{A}_{N,B}(h_1,h_2;\alpha,\beta)} \frac{1}{n^{1+im_2v}}\right|\right)\\  
&-\sum_{1 \leq |m_1|,|m_2| \leq K} \frac{1}{|m_1m_2|}\left|\sum_{n \in \mc{A}_{N,B}(h_1,h_2;\alpha,\beta)} \frac{1}{n^{1+i(m_1u + m_2v)}}\right| - O\left(\frac{\log x}{K}\right).
\end{align*}
Now, as $m_1u \neq -m_2v$ for all $1 \leq |m_1|,|m_2| \leq K$, we can estimate each of the sums in absolute value above via Lemma \ref{TWIST} as before. If $x$ is sufficiently large (in terms of $\left(\min_{|m_1|,|m_2|\leq K}|m_1u+m_2v|\right)^{-1}$) the claim of part b) follows just as that of part a) did.
\end{proof}
\subsection{Reduction to Pseudo-pretentious Functions}
Let $f,g : \mb{N} \ra \mb{T}$ be completely multiplicative functions. Assume, as per the hypotheses of Theorem \ref{EKCFULL}, that $\overline{\{(f(n),g(n+1))\}_n}\ne \mb{T}^2.$ Using Theorem \ref{TAOBINTHM}, we shall show in this subsection that $f$ and $g$ must both be pseudo-pretentious.
\begin{prop} \label{REDUCE}
Let $f,g : \mb{N} \ra \mb{T}$ be completely multiplicative functions. Suppose $\{(f(n),g(n+1))\}_n$ is such that there is an $\e > 0$ and a pair $(z,w) \in \mb{T}^2$ for which\footnote{For a pair $\mbf{z} \in \mb{C}^2$, we write $\|\mbf{z}\|_{\ell^1} := |z_1| + |z_2|$.} 
$$\|(f(n),g(n+1))-(z,w)\|_{\ell_1} \geq \e \text{ for all $n$ sufficiently large}.$$ 
Then there exist primitive Dirichlet characters $\chi_1,\chi_2$ to respective moduli $q_1,q_2 = O_{\e}(1)$, minimal positive integers $k,l = O_{\e}(1)$, real numbers $t_1,t_2 = O_{\e}(1)$ and completely multiplicative functions $h_1,h_2 : \mb{N} \ra \mb{T}$ such that: 
\begin{enumerate}[(i)]
\item$\mb{D}(f,h_1n^{it_1};x),\mb{D}(g,h_2n^{it_2};x) \ll_{\e} 1$,
\item $h_1^k = \tilde{\chi}_1$ and $h_2^l = \tilde{\chi}_2$.
\end{enumerate}
\end{prop}
This is a two-dimensional analogue of the reduction argument used at the beginning of Section 2 of \cite{RIG}. To appropriately formalize the arguments leading to Proposition \ref{REDUCE}, we recall the following (essentially standard) definitions.
\begin{def1}
Let $d,N \geq 1$ and let $\{a_n\}_n \subset \mb{T}^d$. The \emph{logarithmic discrepancy (of height $N$)} of $\{a_n\}_n$ is the quantity
\begin{equation*}
D_N(\{a_n\}_n) := \sup_{I_1,\ldots,I_d \subseteq \mb{T} \atop I_j \text{ arcs}} \left|\frac{1}{\log N}\sum_{n \leq N \atop a_n \in B(\mbf{I})} \frac{1}{n} - \prod_{1 \leq j \leq d} |I_j|\right|,
\end{equation*}
where $B(\mbf{I}) := \prod_{1 \leq j \leq d} I_j$. Similarly, we let $D_N^{\ast}(\{a_n\}_n)$ denote \footnote{By the triangle inequality, it is easy to check that if $\{a_n\}_n \subset \mb{T}^d$ then
\begin{equation}\label{DISCPROPERTY}
D_N^{\ast}(\{a_n\}_n) \leq D_N(\{a_n\}_n) \leq 2^dD_N^{\ast}(\{a_n\}_n).
\end{equation}
We will use this relationship between $D_N$ and $D_N^{\ast}$ below.} 
the same quantity but where the sup is over all $d$-tuples of arcs $I_j$ all of whose left endpoints are 1.
\end{def1}
\begin{def1}
A sequence $\{a_n\}_n \subset \mb{T}^d$ is \emph{logarithmically equidistributed} if $D_N(\{a_n\}_n) = o(1)$, as $N \ra \infty$. 
\end{def1}
It is easy to see that logarithmically equidistributed sequences in $\mb{T}^d$ are also dense in $\mb{T}^d.$ 
Before launching into the proof, we must exclude the following degenerate case from consideration. 
\begin{prop} \label{NOZEROH}
Let $f,g: \mb{N} \ra \mb{T}$ be completely multiplicative. Suppose exactly one of $f$ and $g$ is pseudo-pretentious. 
Then $\overline{\{(f(n),g(n+1))\}_n}=\mb{T}^2.$
\end{prop}
To prove Proposition \ref{NOZEROH}, we shall need a concentration estimate , showing that if a completely multiplicative function $F: \mb{N} \ra \mb{T}$ is 1-pretentious then $F(n)$ is roughly constant for \emph{most} $n \leq x$. This will be used in the proof of Lemma \ref{CHEB} below, which will allow us to approximate pseudo-pretentious functions pointwise on a positive upper density subset of integers $n$ with $P^-(n(Bn+1)) > N$ (with $B$ and $N$ fixed).  \\
Given a completely multiplicative function $h : \mb{N} \ra \mb{T}$ and $N \geq 1$, set $$\mf{I}_h(x;N) := \sum_{N < p \leq x} \frac{\text{Im}(h(p))}{p}.$$
\begin{lem} \label{SMALLDEV}
Let $N,B \geq 1$ with $2|B$. Then as $x \ra \infty$, 
\begin{equation*}
\sum_{n \leq x \atop P^-(n(Bn+1)) > N} \left|f(n)-e^{i\mf{I}_f(x;N)}\right|^2 \ll \Phi_{N,B}(x)\left(\mb{D}(f,1;N,x)^2+ \frac{1}{N}\right) + 4^{\pi(N)}x\frac{\log_2 x}{\log x}.
\end{equation*}
The same estimate holds when $f(n)$ is replaced by $f(Bn+1)$.
\end{lem}\begin{proof}
This is a simple extension of Proposition 2.3 in \cite{Klu}, but we give the details for the reader's convenience. \\
The claim is trivial if $\mb{D}(f,1;N,x) \geq 1$. Thus, in what follows we shall assume that $\mb{D}(f,1;N,x) < 1$. Define an additive function $h: \mb{N} \ra \mb{C}$ on prime powers via $h(p^k) := f(p^k)-1.$ For each $n \in \mb{N}$,
$$f(n) = \prod_{p^k||n} f(p^k) = \prod_{p^k||n} \left(1+h(p^k) \right) = e^{h(n)} + O\left(2\sum_{p^k||n} (1-\text{Re}(f(p^k)))\right).$$
Summing over all $n \leq x$ with $P^-(n(Bn+1)) > N$, we get
\begin{align*}
\sum_{n \leq x \atop P^-(n(Bn+1)) > N} |f(n)-e^{i\mf{I}_f(x;N)}|^2 &\ll \sum_{n \leq x\atop P^-(n(Bn+1))>N} |e^{h(n)} - e^{i\mf{I}_f(x;N)}|^2 \\
&+ \sum_{p^k \leq x} (1-\text{Re}(f(p^k)))\sum_{n \leq x \atop P^-(n(Bn+1)) > N} 1_{p^k||n} =: T_1 + T_2.
\end{align*}
Using Lemma \ref{TRIV}, we have
\begin{align*}
T_2 &= \sum_{p^k \leq x \atop p >N} (1-\text{Re}(f(p^k))) \left(\sum_{m \leq x/p^k \atop P^-(m(Bp^km+1)) > N} 1-\sum_{m \leq x/p^{k+1} \atop P^-(m(Bp^{k+1}m+1))>N} 1\right) \\
&= x\prod_{p'|B} \left(1-\frac{1}{p'}\right) \prod_{3 \leq p'' \leq N} \left(1-\frac{2}{p''}\right) \sum_{p^k \leq x \atop p > N} \frac{1-\text{Re}(f(p^k))}{p^k}\left(1-\frac{1}{p}\right)^{1+1_{p\nmid B}} + O\left(4^{\pi(N)}\pi(x)\right) \\
&\ll \Phi_{N,B}(x) \left(\mb{D}(f,1;N,x)^2 + \frac{1}{N}\right) + 4^{\pi(N)}\frac{x}{\log x}.
\end{align*}
We next treat $T_1$. Define
$$\mu_h(x) := \sum_{p^k \leq x \atop p > N} \frac{h(p^k)}{p^k}(1-1/p)^{1+1_p\nmid B},$$
which arises as the mean value of $h(n)$ in the estimate
$$\sum_{n \leq x \atop P^-(n(Bn+1)) > N} h(n) = \mu_h(x)\Phi_{N,B}(x) + O\left(4^{\pi(N)}\pi(x)\right),$$
using the same argument as was used to bound $T_2$. Since $\text{Re}(h(n)) \leq 0$ for all $n$, it follows that $\text{Re}(\mu_h(x)) \leq 0$, and we thus have
\begin{equation}\label{EMUH}
\sum_{n \leq x \atop P^-(n(Bn+1))>N} |e^{h(n)} - e^{\mu_h(x)}|^2 \leq \sum_{n \leq x \atop P^-(n(Bn+1))>N} |h(n)-\mu_h(x)|^2.
\end{equation}
Following the usual proof of the Tur\'{a}n-Kubilius inequality (e.g., as in Section III.2 of \cite{Ten2}), we have
\begin{align*}
&\sum_{n \leq x \atop P^-(n(Bn+1))>N} |h(n)-\mu_h(x)|^2 = \sum_{n \leq x \atop P^-(n(Bn+1)) > N} |h(n)|^2 - \Phi_{N,B}(x)|\mu_h(x)|^2 + O\left(4^{\pi(N)}|\mu_h(x)|\pi(x)\right) \\
&= \sum_{n \leq x \atop P^-(n(Bn+1)) > N} \sum_{p^k,q^l||n} h(p^k)\bar{h(q^l)} - \Phi_{N,B}(x)|\mu_h(x)|^2 + O\left(4^{\pi(N)}|\mu_h(x)|\pi(x)\right).
\end{align*}
Denote by $\Sigma$ the double sum in this last expression. Splitting the terms $p =q$ from $p \neq q$, $\Sigma$ can be estimated with Lemma \ref{TRIV} to give
\begin{align*}
\Sigma &= \sum_{p^kq^l \leq x \atop p \neq q, p,q > N} h(p^k)\bar{h(q^l)} \sum_{m \leq x/(p^kq^l) \atop P^-(m(Bp^kq^lm+1))>N} 1 + \sum_{p^k \leq x \atop p > N} |h(p^k)|^2\sum_{m \leq x/p^k \atop P^-(m(Bp^km+1))>N} 1 \\
&= \Phi_{N,B}(x)\sum_{p^k q^l \leq x \atop p \neq q, p,q > N} \frac{h(p^k)\bar{h(q^l)}}{p^kq^l} \left(1-\frac{1}{p}\right)^{1+1_{p\nmid B}}\left(1-\frac{1}{q}\right)^{1+1_{q\nmid B}} \\
&+ \Phi_{N,B}(x) \sum_{p^k \leq x \atop p > N} \frac{|h(p^k)|^2}{p^k}\left(1-\frac{1}{p}\right)^{1+1_{p\nmid B}} + O\left(4^{\pi(N)}\pi_2(x)\right),
\end{align*}
Furthermore, 
\begin{align*}
&\sum_{p^kq^l \leq x \atop p \neq q, p,q >N} \frac{h(p^k)\bar{h(q^l)}}{p^kq^l}\left(1-\frac{1}{p}\right)^{1+1_{p\nmid B}}\left(1-\frac{1}{q}\right)^{1+1_{q\nmid B}} \\
&= |\mu_h(x)|^2 + O\left(\sum_{p^k,p^l \leq x \atop p > N} \frac{|h(p^k)||h(p^l)|}{p^{k+l}} + \sum_{p^k,q^l \leq x \atop p^kq^l > x, p,q >N} \frac{|h(p^k)||h(q^l)|}{p^kq^l}\right) \\
&= |\mu_h(x)|^2 + O\left(\sg_h(x)^2 + x^{-1/2}\sg_h(x)^2\right),
\end{align*}
where we defined $$\sg_h(x)^2 := \sum_{p^k \leq x \atop p >N} \frac{|h(p^k)|^2}{p^k}\left(1-\frac{1}{p}\right)^{1+1_{p\nmid B}},$$ and used Cauchy-Schwarz in the last line. It follows that 
$$\Sigma \ll \Phi_{N,B}(x)\sg_h(x)^2 + 4^{\pi(N)} \pi_2(x) \ll \Phi_{N,B}(x)\sg_h(x)^2 + 4^{\pi(N)} x\frac{\log_2 x}{\log x},$$ 
and hence
\begin{align}
\sum_{n \leq x \atop P^-(n(Bn+1)) > N} |h(n)-\mu_h(x)|^2 &\ll \sg_h(x)^2\Phi_{N,B}(x) + 4^{\pi(N)}\left(x\frac{\log_2 x}{\log x} + |\mu_h(x)|\pi(x)\right) \nonumber\\
&\ll \left(\mb{D}(f,1;N,x)^2+\frac{1}{N}\right)\Phi_{N,B}(x) + 4^{\pi(N)}x\frac{\log_2 x}{\log x},\label{TKHERE}
\end{align}
using $|\mu_h(x)| \leq 2\sum_{p^k \leq x} 1/p^k \leq 2\log_2 x + O(1)$ and the estimate
$$\sg_h(x)^2 \leq \sum_{p^k \leq x \atop p>N} \frac{|1-f(p^k)|^2}{p^k} \ll \sum_{p^k \leq x \atop p>N} \frac{1-\text{Re}(f(p^k))}{p^k} = \mb{D}(f,1;N,x)^2 + \frac{1}{N}$$
in the last line. Inserting \eqref{TKHERE} into \eqref{EMUH}, and using this in the definition of $T_1$, we get
\begin{align} 
&\sum_{n \leq x \atop P^-(n(Bn+1)) > N} |e^{h(n)}-e^{i\mf{I}_f(x;N)}|^2 \ll \sum_{n \leq x \atop P^-(n(Bn+1)) > N} |e^{h(n)}-e^{\mu_h(x)}|^2 + \sum_{n \leq x \atop P^-(n(Bn+1))>N} |e^{\mu_h(x)} - e^{i\mf{I}_f(x;N)}|^2 \nonumber\\
&\ll \Phi_{N,B}(x)|e^{\mu_h(x)} - e^{i\mf{I}_f(x;N)}|^2+\Phi_{N,B}(x)\left(\mb{D}(f,1;N,x)^2 + \frac{1}{N}\right) + 4^{\pi(N)} x\frac{\log_2 x}{\log x}. \label{NEARLY}
\end{align}
Moreover, combining all of the above we get
$$|e^{i\mf{I}_f(x;N)}-e^{\mu_h(x)}|^2 \ll \frac{1}{N} + \left|e^{i\mf{I}_f(x;N)}-e^{i\mf{I}_f(x;N)-\sum_{N < p \leq x} \frac{1-\text{Re}(f(p))}{p}}\right|^2 \ll \frac{1}{N} + \mb{D}(f,1;N,x)^2.$$
Coupled with \eqref{NEARLY}, this completes the proof of the first assertion of the lemma.\\
The assertion with $f(Bn+1)$ in place of $f(n)$ is proved similarly, and we leave the details for the reader. 
\end{proof}
\begin{lem}\label{CHEB}
Let $\eta > 0$ and $C \geq 1$. Let $B \geq 1$ be an even integer. Suppose furthermore that $f: \mb{N} \ra \mb{T}$ is a pseudo-pretentious multiplicative function, with $f$ pretending to be $h(n)n^{it}$ with $h_1$ a pseudocharacter and $t \in \mb{R}$. Then there is an infinite sequence $\{x_j\}_j$ of positive real numbers and a large parameter $N$ (depending at most on $\eta$) such that if $j = j(N)$ is chosen sufficiently large then
$$f(n) = h(n)n^{it} + O(\eta)$$ 
for all but $O_C(\eta^C\Phi_{N,B}(x_j))$ choices of $n \leq x_j$ with $P^-(n(Bn+1)) > N$.
\end{lem}
\begin{proof}
We consider several cases, according to the behaviour of the series 
$$\mf{I}_f(\infty) := \lim_{x \ra \infty} \mf{I}_f(x;1) = \sum_p \frac{\text{Im}(f(p))}{p}.$$
First, suppose $\mf{I}_f(\infty)$ converges absolutely. Then, choosing $N$ large enough in terms of $\eta$, it follows that $\mf{I}_f(x;N) \ll \eta$ for all $x > N$. In this case, we shall choose $\{x_j\}_j$ to be the set of all $x > N$. \\
Next, suppose that $\mf{I}_f(\infty)$ diverges. In this case, since $\mf{I}_f(n+1;N)-\mf{I}_f(n;N)\to 0$ as  $n\to\infty,$ it follows that the sequence of fractional parts of $\mf{I}_f(x;N)/2\pi$ must be dense in $[0,1]$, for any $N$. In this case, we may choose any large $N$ and $\{x_j\}_j$ to be a sequence for which $\mf{I}_f(x_j;N)/2\pi \ra 0$, as $j \ra \infty$.\\
Lastly, suppose $\mf{I}_f(\infty)$ converges conditionally. Let $\alpha$ be a limit point of $\mf{I}_f(x;1)$, and choose $\{x_j\}_j$ to be a sequence for which $\mf{I}_f(x_j;N) \ra \alpha$. Picking $j_0$ sufficiently large in terms of $\eta$, then setting $N := x_{j_0}$, it follows that $\mf{I}_f(x_j;N) \ll \eta$ for large $j$. \\  
Now, let $F(n) := f(n)\bar{h(n)}n^{-it}$ for each $n$. As $F$ is $1$-pretentious, we can assume $N$ is chosen large enough in terms of $\eta$ (as in the analysis above) so that $\mb{D}(F,1;N,\infty)^2 + \frac{1}{N} \ll \eta^{C+2}.$ Furthermore, as discussed above we have selected a sequence $\{x_j\}_j$ such that $\mf{I}_f(x_j;N)/2\pi \ (\text{mod } 1) \ra 0$. Taking $j$ sufficiently large in terms of $N$, Lemma \ref{SMALLDEV} and Chebyshev's inequality give
\begin{align}\label{EXCEP}
&\left|\left\{n \leq x_j : P^-(n(Bn+1)) > N, |F(n) - e^{i\mf{I}_f(x_j;N)}| > \frac{\eta}{2}\right\}\right| \leq 4\eta^{-2} \sum_{n \leq x \atop P^-(n(Bn+1)) > N} |F(n)-e^{i\mf{I}_f(x_j;N)}|^2\nonumber\\
&\ll \eta^{-2}\Phi_{N,B}(x_j)\left(\mb{D}(f,1;N,x_j)^2+\frac{1}{N}\right) + \eta^{-2}4^{\pi(N)}x_j\frac{\log_2 x_j}{\log x_j} \ll \Phi_{N,B}(x_j) \eta^C.
\end{align}
For all other $n \leq x_j$ with $P^-(n(Bn+1)) > N$ we choose $j$ larger if necessary so that $\|\mf{I}_f(x_j;N)/2\pi\| < \eta/2$, and hence
$$|f(n)-h(n)n^{it}| = |F(n) -1| \leq |F(n)-e^{i\mf{I}_f(x_j;N)}| + |e^{i\mf{I}_f(x_j;N)}-1| < \eta/2 + \eta/2 = \eta,$$ 
which implies the claim.
\end{proof}
\begin{proof}[Proof of Proposition \ref{NOZEROH}]
Suppose there is an $\e > 0$ and a pair $(z,w) \in \mb{T}^2$ such that\footnote{Here and elsewehere, we write $B_{\e}((z,w))$ to denote the $\e$-ball about $(z,w)$ in $\mb{C}^2$ with respect to the $\ell^1$-norm.} $(f(n),g(n+1)) \notin B_{\e}((z,w))$ for all large $n$.  Let $h$ be a pseudocharacter such that $f$ is $hn^{it}$-pretentious, and assume it has order $kr$ (i.e., $h: \mb{N} \ra \mu_{kr}$). Select $\{n_j\}_j$ on which $f(2n_j) \ra z$, and hence for $j$ chosen large enough we have $f(2n_j) = z + O(\e^2)$. By Lemma \ref{SMALLDEV}, we can choose $N$ sufficiently large in terms of $\e$ and a sequence $\{x_{\ell}\}_{\ell}$ such that for large enough $\nu$,
$$f(2n_jm) = zh(m)m^{it} + O(\e^2),$$
for all but $O(\e^3\Phi_{N,2n_j}(x_{\ell})$ integers $m \leq x_{\ell}$ with $P^-(m(2n_jm+1)) > N.$ 
Thus, if $\e$ is sufficiently small then for all but a small number of exceptions, we have 
$$(h(m)m^{it},g(2n_jm+1)) \notin B_{2\e/3}((1,w)).$$ 
To yield a contradiction, we shall presently show that there exists a sufficiently dense subset of integers $m \leq x_{\ell}$ that satisfies the following properties: 
\begin{enumerate}[(i)]
\item $h(m) = 1$, \item $P^-(m(2n_jm+1)) > N$, \item $|g(2mn_j+1)- w| < \e/4$, and \item $|m^{it} - 1| < \e/4$.
\end{enumerate}
The proof of this argument is similar to that of Proposition \ref{ONES}. For convenience, put $X := x_{\ell}$, for some sufficiently large index $\nu$ (depending at most on $\e$). Let $I$ be a symmetric arc of length $\e/4$ about $1$, and let $J$ be a symmetric arc of the same length about $w$. Then 
\begin{align*}
\mf{G} &= \sum_{m \leq X \atop P^-(m(2n_jm+1)) > N} \frac{1_I(m^{it})1_{h_1(m) = 1} 1_{J}(g(mn_j+1))}{m}.
\end{align*}
We write the arcs $I$ and $J$ as the images of intervals $[a,b]$ and $[c,d]$ (say) in $\mb{R}$ under the map $t \mapsto e(t)$, such that $[a,b]$ contains an integer (so that $I$ contains $1$). We consider two cases, depending on $t$. \\
First, if $t = 0$ then $1_I(m^{it}) = 1$ for all $m.$ Using the properties of Beurling polynomials, we have minorization
\begin{equation}\label{MINORFORJ}
1_J(g(2n_jm+1)) \geq \left(|J| - \sum_{0 \leq |l| \leq K} \hat{B}_K(l)\left(e(bl)g(2n_jm+1)^{-l} + e(-al)g(2n_jm+1)^l\right)\right),
\end{equation}
for any parameter $K = K(X)$ to be chosen. 
Inserting this into the definition of $\mf{G}$ and summing over $m \leq X$ with the given conditions, we get
\begin{align*}
\mf{G} &\geq |J|\sum_{m \leq X \atop P^-(m(2n_jm+1)) > N} \frac{1_{h_1(m) = 1}}{m} - \sum_{1 \leq |l| \leq K} \hat{B}_K(l) \left(e(dl)M_{-l}(X) + e(-cl)M_l(X)\right),
\end{align*}
where for any $l \in \mb{Z}$ we have put
$$M_l(X) := \sum_{m \leq X \atop P^-(m(2n_jm+1)) > N}\frac{1_{h_1(m)=1}g(2n_jm+1)^{-l}}{n}.$$
Expressing $1_{h_1(m) = 1} = \frac{1}{kr} \sum_{\nu (kr)} h_1(m)^{\nu}$, we have
$$M_l(X) = \frac{1}{kr} \sum_{\nu (kr)} \sum_{m \leq X \atop P^-(m(2n_jm+1)) > N}\frac{h_1(m)^{\nu}g(2n_jm+1)^{-l}}{n} =: \frac{1}{kr} \sum_{0 \leq \nu \leq kr-1} M_{l}(X;\nu),$$
for any $l \in \mb{Z}$. Inserting this into our lower bound for $\mf{G}$, we find
\begin{align}
\mf{G} &\geq \frac{1}{kr} \sum_{0 \leq \nu \leq kr-1} |J|\sum_{m \leq X \atop P^-(m(2n_jm+1)) > N} \frac{h_1(m)^{\nu}}{m} \nonumber\\
&-  \frac{1}{kr} \sum_{0 \leq \nu \leq kr-1} \sum_{1 \leq |l| \leq K} \hat{B}_K(l) \left(e(dl)M_{-l}(X;\nu) + e(-cl)M_{l}(X;\nu)\right) \nonumber\\
&=: \frac{1}{kr} \sum_{0 \leq \nu \leq kr-1} \mc{S}_{\nu}. \label{ALLS}
\end{align}
Fix $0 \leq \nu \leq kr-1$. Now, by hypothesis, $g$ is not pseudo-pretentious, so that $g^{l}1_{P^- > N}$ is non-pretentious for all $l \in \mb{Z} \bk \{0\}$. Hence, Theorem \ref{TAOBINTHM} implies that 
$$M_{l}(X;\nu) = \sum_{m \leq X} \frac{(h1_{P^->N})(m)^{\nu}(g1_{P^->N})(2n_jm+1)^l}{m} = o(\log X),$$
for all $1 \leq |l| \leq K$. Multiplying by $1/|l|$ and summing over $1 \leq l\leq K$ in the above estimate, we get
\begin{equation*}
\mc{S}_{\nu} = |J|\sum_{m \leq X \atop P^-(m(2n_jm+1)) > N} \frac{h_1(m)^{\nu}}{m} + o((\log x) (\log K)) + O\left(\frac{\log x}{K}\right),
\end{equation*}
for all $0 \leq \nu \leq kr-1$. Now, if $1 \leq \nu \leq kr-1$ then as $h$ is a pseudocharacter of order $kr$, $h^{\nu}1_{P^->N}$ is \emph{not} $n^{iu}$-pretentious for any (fixed) choice of $u \in \mb{R}$. Thus, Hal\'{a}sz' theorem implies that 
$$\sum_{m \leq X \atop P^-(m(2n_jm+1)) > N} \frac{h_1(m)^{\nu}}{m} = o(\log X)$$
as well. It follows that when $\nu \neq 0$, we have 
$$S_{\nu} = o((\log X)(\log K)) + O\left(\frac{\log X}{K}\right) = o(\log X),$$
if $K = K(X)$ is chosen to be tending to infinity with $X$ sufficiently slowly, and therefore
$$\mf{G} \geq \frac{|J|}{kr}\sum_{m \leq X \atop P^-(m(2n_jm+1)) > N} \frac{1}{m} - o(\log x).$$
Next, suppose $t \neq 0.$ Arguing as before, this time invoking a minorization like \eqref{MINORFORJ} with $1_I(n^{it})$ as well and using the triangle inequality, we have (cf. \eqref{ALLS} and the lines preceding it)
\begin{align*}
\mf{G} &\geq \frac{1}{kr} \sum_{0 \leq \nu \leq kr-1} |I||J| \sum_{m \leq X \atop P^-(m(2n_jm+1)) > N} \frac{h_1(m)^{\nu}}{m} \\
&- \frac{1}{kr} \sum_{0 \leq \nu \leq kr-1} \sum_{0 \leq |l_1|, |l_2| \leq m \atop \max\{|l_1|,|l_2|\} \geq 1} \frac{1}{R(l_1,l_2)} \sum_{u,v \in \{-1,+1\}} |M_{ul_1,vl_2}(X;\nu)| \\
&=: \frac{1}{kr}\sum_{0 \leq \nu \leq kr-1} T_{\nu},
\end{align*}
where $R(l_1,l_2) := \max\{1,|l_1|\}\max\{1,|l_2|\}$ whenever $l_1l_2 \neq 0$, and we have set
$$M_{l_1,l_2}(X;\nu) := \sum_{m \leq X \atop P^-(m(2n_jm+1)) >N} \frac{h_1(m)^{\nu} g(2n_jm+1)^{l_1}m^{il_2t}}{m}.$$
Since $n \mapsto g(n)^{l_2}1_{P^-(n)>N}n^{il_1t}$ is non-pretentious for all $l_2 \neq 0$, and $n\mapsto h(n)^{\nu}1_{P^-(n)>N}n^{-il_1t}$ is non-pretentious whenever $\max\{|l_1|,|l_2|\} \geq 1$, Theorem \ref{TAOBINTHM} implies that $M_{l_1,l_2}(X;\nu) = o((\log x))$ for all such $l_1,l_2$ and any $\nu$. In particular,
\begin{align*}
T_{\nu} &= |I||J|\sum_{m \leq X \atop P^-(m(2n_jm+1)) > N} \frac{h(m)^{\nu}}{m} + o\left((\log X)\sum_{1 \leq l_1,l_2\leq K} \frac{1}{l_1l_2}\right) \\
&= |I||J|\sum_{m \leq X \atop P^-(m(2n_jm+1)) > N} \frac{h(m)^{\nu}}{m} + o((\log X)(\log K)^2).
\end{align*}
By the arguments above, if $\nu \neq 0$ then the sum of $h^{\nu}$ above is $o(\log x)$ by Hal\'{a}sz' theorem. Hence, we get 
$$\mf{G} \geq |I||J| \sum_{m \leq X \atop P^-(m(2n_jm+1)) > N} \frac{1}{m} - o(\log x),$$
for $K=K(X) \ra \infty$ sufficiently slowly.\\
Invoking Lemma \ref{TRIV} (with $q = 1$) and partial summation, we get that when $X$ is sufficiently large relative to $N$, we get
\begin{equation*}
\sum_{m\leq X \atop P^-(m(2n_jm+1)) > N} \frac{1}{m} = \delta_{N,B}\log (X) + o(\log x), \end{equation*}
upon invoking Lemma \ref{TRIV} (with $q = 1$) and partial summation. Since $|I| \geq |I||J| \gg \e^2$, this shows that 
$$\mf{G} \gg \frac{\e^2}{kr} \frac{\Phi_{N,B}(X)}{X}\log X.$$
Hence, assuming (as we may) that $\e \ll (kr)^{-2}$, the set of $n = 2mn_j$ with $h_1(m) = 1$, $|m^{it}-1| < \e/4$, $|g(mn_j+1)-w| < \e/4$ and $P^-(m(2n_jm+1)) > N$ intersects in a set of positive upper density with the set of $n \leq 2n_jX$ where $f(m) = h_1(m)m^{it} + O(\e^2)$. \\
It thus follows that for a set of integers $n$ with positive upper density, we have 
$$|f(n)-z| = |f(m)-1| + O(\e^2) = |h(m)m^{it}-1| + O(\e^2)= |m^{it}-1| + O(\e^2) < \e/3$$
and
$|g(n+1)-w|= |g(2mn_j+1)-1| < \e/4.$ Consequently, $(f(n),g(n+1)) \in B_{2\e/3}((z,w))$ which contradicts our initial assumption. This contradiction completes the proof.
\end{proof}
Having established the above proposition, we now know that if $\overline{\{(f(n),g(n+1))\}_n} \neq \mb{T}^2$ then either both $f$ and $g$ are pseudo-pretentious, or neither is. To complete the proof of Proposition \ref{REDUCE}, therefore, we shall show that the latter case is not possible.
\begin{proof}[Proof of Proposition \ref{REDUCE}] 
Let us assume for the sake of contradiction that one of $f$ and $g$ are not pseudo-pretentious. Proposition \ref{NOZEROH} implies that, then, both $f$ \emph{and} $g$ are not pseudo-pretentious. Let $a,b: \mb{N} \ra \mb{R}/\mb{Z}$ be completely additive functions for which $f(n) = e(a(n))$ and $g(n) = e(b(n))$. For each $M \in \mb{N}$ let $$F_M(x_1,x_2) := \frac{1}{\log M} \sum_{n \leq M \atop (a(n),b(n+1)) \in [0,x_1]\times [0,x_2]} \frac{1}{n},$$ and $G_M(x_1,x_2) = x_1x_2$ identically for all $M$. Applying Theorem 2 of \cite{NiP} gives that for any $K \geq 1$ we have
\begin{align}\label{ETAPP}
D_M^{\ast}(\{(f(n),g(n+1))\}_n) &= \sup_{x_1,x_2 \in [0,1)}|F_M(x_1,x_2) - G_M(x_1,x_2)| \\
&\ll \frac{1}{K+1} + \frac{1}{\log M} \sum_{0 \leq |m_1|,|m_2| \leq K \atop (m_1,m_2) \neq (0,0)}\frac{1}{R(m_1,m_2)}\left|\sum_{n \leq M} \frac{f(n)^{m_1}g(n+1)^{m_2}}{n}\right| ,
\end{align}
where $R(m_1,m_2) = \max\{1,|m_1|\}\max\{1,|m_2|\}$ whenever $m_1m_2 \neq 0$. \\
Now, by hypothesis, there is an $\e > 0$ and a point $(z,w) \in \mb{T}^2$ such that $(f(n),g(n+1)) \notin B_{\e}((z,w))$ for all large $n$. In particular, this means that
\begin{align*}
D_M(\{f(n),g(n+1)\}_n) &= \sup_{I,J \subset \mb{T} \atop \text{arcs}} \left|\frac{1}{\log M}\sum_{n \leq M \atop (f(n),g(n+1)) \in I\times J} \frac{1}{n}-|I||J|\right| \\
&\geq \left|\frac{1}{\log M} \sum_{n \leq M \atop (f(n),g(n+1)) \in I\times J} \frac{1}{n}-B_{\e}((z,w))\right| = |I||J| \gg \e^2.
\end{align*}
From \eqref{DISCPROPERTY}, we get $D_M^{\ast}(\{(f(n),g(n+1))\}_n) \gg \e^2.$ Combining this with \eqref{ETAPP} and choosing $K = \llf\frac{C}{\e}\rrf$ with $C > 0$ a sufficiently large constant, the pigeonhole principle implies that for some pair $(m_1,m_2) \neq (0,0)$ with $|m_1|,|m_2| \leq K$
\begin{equation*}
\left|\sum_{n \leq M} \frac{f(n)^{m_1}g(n+1)^{m_2}}{n}\right| \gg_{\e} \log M.
\end{equation*}
Note that since neither $f$ nor $g$ is assumed to be pseudo-pretentious, we must have $m_1m_2 \neq 0$ as otherwise Hal\'{a}sz' theorem would yield
$$\left|\sum_{n \leq M} \frac{f(n)^{m_1}}{n}\right|,\left|\sum_{n \leq M} \frac{g(n)^{m_2}}{n}\right| = o(\log M).$$
Thus, we may apply Proposition \ref{ULTPRETLEM}, 
which gives that $f$ and $g$ are indeed \emph{both} pseudo-pretentious and the result follows.
\end{proof}
\section{Proof of Proposition \ref{RATRATIO}} \label{SECRATRATIO}
Before addressing Propositions \ref{EVRAT} and \ref{IRRAT}, we pause to prove Proposition \ref{RATRATIO}, as its proof will be related to several subcases of Proposition \ref{IRRAT} especially. \\
Write $f(n) = f_0(n) n^{it}$ and $g(n) = g_0(n)n^{it'}$, where $f_0$ and $g_0$ are completely multiplicative functions taking values in $\mu_k$ and $\mu_l$, respectively. We show in this section that if $\overline{\{f(n),g(n+1)\}_n} \neq \mb{T}^2$ then $t/t' \in \mb{Q}$. \\
To that end, we will need the following lemma, which is a standard extension of a ``repulsion'' estimate for the pretentious distance~\cite{GranSound}, due to Granville and Soundararajan (see, for instance, Lemma C.1 in \cite{MRT} for the case $k = 2$). We recall that for arithmetic functions $F,G : \mb{N} \ra \mb{U}$, where $\mb{U}$ denotes the closed unit disc, and $x \geq 2$, we set
$$\mb{D}(F,G;x) := \left(\sum_{p \leq x} \frac{1-\text{Re}(F(p)\bar{G(p)})}{p}\right)^{\frac{1}{2}}.$$
\begin{lem} \label{EQUIV}
Let $k \geq 1$ and let $f : \mb{N} \ra \mu_k$ be a multiplicative function. Then 
\begin{equation*}
\inf_{|t| \leq x} \mb{D}(f,n^{it};x) \geq \frac{1}{2k}\min\left\{\sqrt{\log_2 x}, \mb{D}(f,1;x)\right\} - O_k(1).
\end{equation*}
\end{lem}
\begin{proof}
The triangle inequality gives
\begin{align}\label{eqq17}
k\mathbb{D}(f,n^{it};x)\geq \mathbb{D}(f^k,n^{itk};x)=\mathbb{D}(1,n^{itk};x).
\end{align}
Now, the Vinogradov-Korobov zero-free region (see, for instance, Section 9.5 of \cite{Mon}) gives that 
$$|\zeta(1+1/\log x + it)| \ll (\log |t|)^{0.67} \text{ for all } 1\leq |t| \leq x^2.$$
It follows that
\begin{align*}
\mb{D}(1,n^{itk};x)^2&= \sum_{p \leq x} \frac{1-\text{Re}(p^{-ikt})}{p} \geq \log_2 x - \log|\zeta(1+1/\log x + itk)| - O(1)\\
&\geq \log_2 x - 0.67 \log_2 |tk| - O(1) \geq 0.33 \log_2 x - O_k(1).
\end{align*}
Inserting this estimate into \eqref{eqq17}, we get that 
\begin{equation}\label{TRIINTRIV}
\mb{D}(f,n^{it};x) \geq \frac{1}{2k} \sqrt{ \log_2 x} - O_k(1) 
\end{equation}
in this case.\\
Suppose now that $|t|\leq 1$. We may assume that
\begin{align}\label{eqq7}
\mathbb{D}(1,n^{itk};x)<\frac{1}{2}\mathbb{D}(f,1;x)
\end{align}
since in the opposite case the lemma follows immediately from \eqref{eqq17}.\\
By the prime number theorem, for each $|u|\leq k$ the asymptotic $\mathbb{D}(1,n^{iu};x)^2=\log(1+|u|\log x)+O_k(1)$ holds. In particular, 
$$\mathbb{D}(1,n^{itk};x)^2= \log(1+k|t|\log x) + O_k(1) = \log(1+|t|\log x) + O_k(1) = \mathbb{D}(1,n^{it};x)^2+O_k(1),$$ 
for $|t|\leq 1$. Applying the triangle inequality once again, then invoking \eqref{eqq17} and \eqref{eqq7}, we find that
\begin{align*}
\mathbb{D}(f,1;x)\leq \mathbb{D}(1,n^{it};x)+\mathbb{D}(f,n^{it};x)&= \mathbb{D}(1,n^{itk};x)+\mathbb{D}(f,n^{it};x)+O_k(1)\\
&< \frac{1}{2}\mb{D}(f,1;x)+k\mathbb{D}(f,n^{it};x)+O_k(1).    
\end{align*}
Combined with \eqref{TRIINTRIV}, this yields the claim.
\end{proof}
The next result shows that for most $n$ with $P^-(n(Bn+1))$ large, we can reduce to the case that $f_0 = h_1$ and $g_0 = h_2$. This will allow us to use the results of previous sections in proving Proposition \ref{RATRATIO}.
\begin{lem} \label{GETHS}
Let $\eta > 0$ be sufficiently small (in terms of $k$ and $l$). Let $B$ be an even integer. Then there is a subsequence $\{x_j\}_j$ and a parameter $N$ depending at most on $\eta$ such that if $j$ is sufficiently large (in terms of $\eta$) then the following holds: \\
for all but $O(\eta \Phi_{N,B}(x_j))$ integers $n \leq x_j$ with $P^-(n(Bn+1)) > N$, we have $f_0(n) = h_1(n)$ and $g_0(Bn+1) = h_2(Bn+1)$.
\end{lem}
\begin{proof}
Since $f(n)^k = n^{it},$ we have
$$\sum_{n \leq x} \frac{f(n)^k\bar{f(n+1)}^k}{n} = \sum_{n \leq x} \frac{1}{n} + O\left(k|t|\right) \gg \log x.$$
A similar estimate holds with $g^l$ in place of $f^k$. Thus, by Proposition \ref{ULTPRETLEM} we have that $f$ and $g$ are both pseudo-pretentious, and that 
there are real numbers $t_1,t_2 \in \mb{R}$, primitive characters $\chi_1$ and $\chi_2$ with conductors $q_1$ and $q_2$ and completely multiplicative functions $h_1,h_2$ such that $h_1(n)^k = \chi_1(n/(n,q_1^{\infty}))$ and $h_2(n)^l  = \chi_2(n/(n,q_2^{\infty}))$, and such that
\begin{align*}
\mb{D}(f_0,h_1n^{i(t_1-t)};x) &= \mb{D}(f,h_1(n)n^{it_1};x) \ll 1 \\
\mb{D}(g_0,h_2n^{i(t_2-t')};x) &= \mb{D}(f,h_1(n)n^{it_1};x) \ll 1.
\end{align*}
Suppose that when $\chi_1$ and $\chi_2$ do not vanish, they take values in $\mu_r$ and $\mu_s$, respectively. We begin by showing that $t_1 = t$ and $t_2 = t'$.
Applying Lemma \ref{EQUIV} with $F = f_0\bar{h_1}$, 
\begin{align*}
1 &\gg \mb{D}(F,n^{i(t_1-t)};x) \geq \inf_{|u| \leq x} \mb{D}(F,n^{iu};x) \geq \frac{1}{2kr}\min\left\{\sqrt{\log_2 x}, \mb{D}(F,1;x)\right\} - O_k(1)\\
&= \frac{1}{2kr}\mb{D}(F,1;x) - O_k(1),
\end{align*}
It follows that $\mb{D}(f_0,h_1;x) \ll_{k,r} 1$, and hence the triangle inequality yields
$$\mb{D}(1,n^{i(t_1-t)};x) \leq \mb{D}(f_0\bar{h_1},1;x) + \mb{D}(f_0\bar{h_1},n^{i(t_1-t)};x) \ll_{k,r} 1.$$
Arguing as in the proof of Lemma \ref{EQUIV} it is clear that for $x$ sufficiently large,
$$\mb{D}(1,n^{i(t_1-t)};x) = \log_2 x - \log|\zeta(1+1/\log x + i(t_1-t))| + O(1) \geq \log_2 x - O\left(1 + \frac{\log x}{1 +|t_1-t|\log x}\right)$$
can only be bounded if $t_1 = t$. Similarly, we must take $t_2 = t$.\\
Now, set $F_0(n) := f_0(n)\bar{h_1(n)}$ and $G_0(n) := g_0(n)\bar{h_2(n)}$. Lemmata \ref{SMALLDEV} and \ref{CHEB} imply that for a suitable choice of $N$ (depending only on $\eta$, $f$ and $g$) and a sequence $\{x_j\}_j$, we get that for large enough $j$,
\begin{align*}
|\{n \leq x_j : P^-(n(Bn+1)) > N,|F_0(n) - 1| > \eta\}|&\ll \eta \Phi_{N,B}(x_j) \\
|\{n \leq x_j : P^-(n(Bn+1)) > N,|G_0(Bn+1) - 1| > \eta\}| &\ll \eta \Phi_{N,B}(x_j).
\end{align*}
It follows that for all but $O(\eta \Phi_{N,B}(x_j))$ integers $n\leq x_j$ with $P^-(n(Bn+1)) > N$, we have that $F_0(n) = 1 + O(\eta)$, and $G_0(Bn+1) = 1 + O(\eta)$, for $N$ sufficiently large. 
But since $F_0$ and $G_0$ take discrete values,  for these $n$, $F_0(n) = G_0(n) = 1$ when $\eta$ is sufficiently small (in terms of $k$ and $l$). In light of the definition of $F_0$ and $G_0$, this implies the claim.
\end{proof}
\begin{prop} \label{RATRATIO}
Suppose that $f(n)^k = n^{it}$ and $g(n)^l = n^{it'}$, for $k,l \in \mb{N}$ and $t,t' \in \mb{R}$ and $tt'\ne 0.$ Then $\overline{\{(f(n),g(n+1))\}}\ne\mb{T}^2$ if, and only if, $t = \lambda t'$ with $\lambda \in \mb{Q}$.
\end{prop}
\begin{proof}
Let $f(n) = f_0(n)n^{it_1}$ and $g(n) = g_0(n)n^{it_2}$, where $f_0^k = 1 = g_0^l.$ As $(n+1)^{it_2} = n^{it_2} + o(|t_2|/n)$, it is equivalent to consider when the sequence of pairs $(f_0(n)n^{it_1},g_0(n+1)n^{it_2})$ is, or is not dense for $n$ sufficiently large. \\
We first suppose that $t/t' =\frac{r_1}{s_1} \in \mb{Q} \bk \{0\}.$ In this case, since 
$$\overline{(f(n)^{ks_1},g(n+1)^{\l r_1})}=\overline{(n^{is_1t},(n+1)^{ir_1t'})}=\{(z,z)\in \mathbb{T}^2\}\ne \mathbb{T}^2,$$ 
we have $\overline{(f(n),g(n+1)}\ne\mb{T}^2.$ \\ 
Suppose now that $t/t' \notin \mb{Q}$, and that $\{(f_0(n)n^{it_1},g(n+1)n^{it_2})\}_n$ is not dense in $\mb{T}^2$. By Proposition \ref{REDUCE}, we know that $f$ and $g$ must be pseudo-pretentious. Thus, let $h_1$ and $h_2$ be pseudocharacters with conductors $q_1$ and $q_2$ respectively, and $t,t' \in \mb{R}$ such that $f$ is $h_1n^{it}$-pretentious and $g$ is $h_2n^{it'}$-pretentious. From the proof of Lemma \ref{GETHS}, it follows that $t = t_1$ and $t' = t_2$. Now, replacing $n$ by $Bn$, where $B = 2q_1q_2$, it follows that 
$$(f_0(n)n^{it_1},g_0(Bn+1)n^{it_2})\notin B_{\e}(z\bar{f(B)}B^{-it},wB^{-it'}).$$ 
Set $(z',w') = (z\bar{f(B)}B^{-it},wB^{-it'})$.
Now, according to Lemma \ref{GETHS} (applied with $\eta = \e^3$), we can choose $x$ sufficiently large (belonging to a prescribed infinite subsequence) and $N$ suitably large in terms of $\e$, such that $f_0(n) = h_1(n)$ and $g_0(Bn+1) = h_2(Bn+1)$, for all but $O\left(\e^3 \Phi_{N,B}(x)\right)$ elements $n \leq x$ with $P^-(n(Bn+1))> N$. Hence, we know that 
$$(h_1(n)n^{it_1},h_2(Bn+1)n^{it_2}) \notin B_{\e}((z',w'))$$ 
for all but $O(\e^3\Phi_{N,B}(x))$ of those $n$. \\
On the other hand by Proposition \ref{ONES} there are $\gg \frac{\e^2}{kl}\Phi_{N,B}(x)$ on which $\max\{|n^{it}-z'|,|n^{it'}-w'|\} \leq \e/4$ and $h_1(n) = h_2(Bn+1) = 1$, whence 
$$(h_1(n)n^{it},h_2(Bn+1)n^{it'}) = (n^{it},n^{it'}) \in B_{\e/4}((z',w')).$$ 
This must intersect with the set of $n$ where $(f_0(n),g_0(Bn+1)) = (h_1(n),h_2(Bn+1))$, provided that $\e$ is sufficiently small (in terms of $k$ and $l$ alone). This contradicts the earlier claim, and proves the proposition.
\end{proof}
\section{The Eventually Rational Case} \label{SECEVRAT}
In this section, we prove Proposition \ref{EVRAT}. That is, we assume that $f$ and $g$ are both eventually rational functions, i.e., for some $N$ sufficiently large there is an $m \in \mb{N}$ such that for all primes $p > N$ we have $f(p)^m = 1$, and that $\{f(n)\}_n$ and $\{g(n)\}_n$ are both dense.\\
Appealing to Proposition \ref{REDUCE} once again, we may assume that $f$ is $h_1 n^{it_1}$-pretentious and $g$ is $h_2n^{it_2}$-pretentious, $k$ and $l$ are minimal positive integers such that $h_1^k = \tilde{\chi_1}$ and $h_2^l = \tilde{\chi_2}$. As above, we let $r$ and $s$ denote the respective orders of $\tilde{\chi_1}$ and $\tilde{\chi_2}$. \\
To prove Proposition \ref{EVRAT}, we need the following technical result, which extends Lemma 2.12 of \cite{RIG} (which is the special case $M' = 1$ and $M'' = P_N := \prod_{p \leq N} p$). In what follows, for a positive integer $n$ we write $\text{rad}(n) := \prod_{p|n}p$ to denote the squarefree kernel of $n$.
\begin{lem} \label{EXT}
Let $N > 2q_1q_2$ be a positive integer, and let $m,m',m''$ be coprime squarefree positive integers such that $P_N= mm'm''$, with $(m',2q_1q_2) = 1$.  Let $M'$ and $M''$ be integers with $\text{rad}(M') = m'$ and $\text{rad}(M'') = m''$ and $M''$ odd, and suppose $M$ is a multiple of $2q_1q_2$ with $\text{rad}(M/2q_1q_2) = m$. Then
\begin{equation*}
\sum_{n \leq x \atop M|n} \frac{1_{n \equiv -\bar{M}(M')}1_{(n(n+1),M'') = 1} 1_{h_1(n) = 1}1_{h_2(n+1) = 1}}{n} = \left(\frac{1}{M\phi(M')klrs}\prod_{p|M''}\left(1-\frac{2}{p}\right) + o(1)\right) \log x.
\end{equation*}
\end{lem}
\begin{proof}
The proof of Lemma \ref{EXT} is similar to that of Lemma 2.12 of \cite{RIG}, and we merely sketch the proof here. If $\mc{S}$ denotes the LHS above then, as $M$,$M'$ and $M''$ are coprime, using orthogonality modulo $M'$ gives
\begin{align*}
\mc{S} &= \frac{1}{krls} \sum_{m \leq x/M \atop mM \equiv 1(M')} \frac{1_{(m(Mm+1),M'') = 1} 1_{h_1(mM) = 1} 1_{h_2(mM+1)= 1}}{mM} \left(\sum_{0 \leq a \leq kr-1}\sum_{0 \leq b \leq ls-1} h_1(mM)^a h_2(mM+1)^b\right) \\
&= \frac{1}{krls}\sum_{\chi(M')} \frac{\chi(M)}{\phi(M')}\sum_{0 \leq a \leq kr-1}h_1(M)^a\sum_{0 \leq b \leq ls - 1} \sum_{m \leq x/M} \frac{\chi(m)h_1(m)^ah_2(Mm+1)^b1_{(n,M'') = 1}1_{(Mm+1,M'')=1}}{mM}.
\end{align*}
One proceeds as in the proof of Lemma 2.12 in \cite{RIG} (using Theorem \ref{TAOBINTHM}) to show that whenever $ab \neq 0$, the contributions to the full sum are $o(\log x)$, irrespective of $\chi$. In the remaining case $a = b = 0$, and $\chi$ is \emph{non-principal} modulo $M'$, the contribution to the full sum is $o(\log x)$ as well. It thus follows that
\begin{align*}
\mc{S} &= \frac{1}{M\phi(M')krls}  \sum_{m \leq x/M} \frac{1_{(m,M''M') = 1}1_{(Mm+1,M'') = 1}}{m} + o(\log x) \\
&= \frac{1}{M\phi(M')krls} \prod_{p|M''} \left(1-\frac{2}{p}\right) \log x + o(\log x),
\end{align*}
the main term arising from the coprimality of $M$ and $M''$ provided that $x$ is large enough in terms of $M$ (using, e.g., the fundamental lemma of the sieve). This proves the claim. 
\end{proof}

\begin{proof}[Proof of Proposition \ref{EVRAT}]
Let $N$ be such that for all $p > N$, we have $f(p)^k = g(p)^l = 1$. 
Assume for contradiction that $(f(n),g(n+1)) \notin B_{\e}(z,w)$, for some $\e > 0$ and some pair $(z,w) \in \mb{T}^2$. By Proposition \ref{REDUCE}, we know that $\mb{D}(f,h_1n^{it};x),\mb{D}(g,h_2n^{it'};x) \ll_{\e} 1$, for some $h_1,h_2$ completely multiplicative taking values in roots of unity. By multiplying the characters $\chi_1$ and $\chi_2$ corresponding to $h_1$ and $h_2$ by the principal character modulo $P_N$, we may assume that $h_1$ and $h_2$ are equal to 1 at all of the primes $p \leq N.$ Furthermore, as $f\bar{h_1}$ and $g\bar{h_2}$ also take values in roots of unity for all but finitely many primes, the proof of Lemma \ref{EQUIV} implies that $t = t' = 0$. \\
Now since $\{f(n)\}_n$ and $\{g(n)\}_n$ are dense, and $f(p)$,$g(p)$ take values in a set of bounded order roots of unity for all but finitely many primes, we know that there is (at least) a prime $p$ and a prime $p'$ at which the argument of $f(p) = e(\alpha)$ and $g(p') = e(\beta)$, with $\alpha,\beta \notin \mb{Q}$. The additional hypothesis of the proposition implies that, in fact, $p \neq p'$. \\
With $\e > 0$ and $z,w \in \mb{T}$ given above, we may apply Kronecker's theorem (with $\alpha$ and with $\beta$, separately) to get $a,b \in \mb{N}$ such that $$\|(f(p)^{a},g(p')^{b}) - (z\bar{f(2q_1q_2)},w)\|_{\ell_1} < \e.$$ Now set $B = 2q_1q_2p^{a}$. Thus, in the remainder of the proof, in order to achieve a contradiction it will suffice to check that we can find $n$ such that 
$f(n) = f(2q_1q_2)f(p)^{a}$ and $g(nB+1) = g(p')^{b}$. \\
Let $\eta > 0$ be a parameter to be chosen depending at most on $\e$, $k$,$r$,$l$ and $s.$ By Lemma \ref{CHEB}, there is a suitable infinite sequence of $x$ (and a possibly larger choice of $N$) for which we have $f(n) = h_1(n)$ and $g(Bn+1) = h_2(Bn+1)$ for all but $O(\eta\Phi_{N,B}(x/B))$ integers $n \leq x$ with $P^-(n(Bn+1)) > N$. Thus, it suffices to show that there are $\gg \e^2\Phi_{N,B}(x/B)$ integers $n \leq x$ with $P^-(n(Bn+1)) > N$ that satisfy:
\begin{enumerate}[(i)]
\item $f(nB/(nB,P_N^{\infty})) = 1 = g((nB+1)/(nB+1,P_N^{\infty}))$; 
\item $(p')^{b}|(nB+1)$ but $(nB(nB+1),P_N/pp') = 2q_1q_2$.
\end{enumerate}
To do this, we shall apply Lemma \ref{EXT}. Let $m = p$, $m' = p'$ and $m'' = P_N/pp'$, and set $M = B$, $M' = (p')^{b}$ and $M''= m''$. 
As $f(nB/(nB,P_N^{\infty})) = h_1(n)$ and $g((Bn+1)/(Bn+1,P_N^{\infty})) = h_2(Bn+1)$ for all but $O(\e^3\Phi_{N,B}(x/B))$, Lemma \ref{EXT} implies that the number of $n \leq x/B$ with the required properties is
$$\gg \frac{1}{B(p')^{b-1}(p'-1)krls}\prod_{p'' \leq N \atop p'' \nmid pp'B} \left(1-\frac{2}{p}\right)\gg \frac{1}{(p')^{b}krls} \Phi_{N,B}(x/B).$$
Choosing $\eta = C((p')^{b}krls)^{-1}$ with a sufficiently small constant $C = C(\e) > 0$ (noting that this quantity depends only on $\e$) implies the claim.
\end{proof}
\section{Some Preliminary Cases of Proposition \ref{IRRAT}} \label{DEGENCASES}
In this section, we shall dispose of several special cases of  Proposition \ref{IRRAT}. In this way, we shall be able to reduce our work in proving Proposition \ref{IRRAT} in Section \ref{SECIRRAT} to focusing on functions with certain prescribed behaviour.\\
In this section, we shall assume that either $f(p)^k = p^{it}$ or $g(p)^l = p^{it'}$ for all \emph{sufficiently large} primes $p$. Since the argument in the first case (i.e., with $f$) is symmetric to that with $g$, we shall focus only on the case with $f$. We consider three subcases of this case, according to the behaviour of $g$: \\
\underline{subcase i):} we assume that $f(p)^k = p^{it}$ for all large primes $p$, and $g$ is an eventually rational function (see Lemma \ref{SKETCH}); \\
\underline{subcase ii):}  we assume that $f(p)^k = p^{it}$ \emph{only} for large primes, and $g(p)^l = p^{it'}$, for \emph{all} primes $p$ (see Lemma \ref{SKETCH2}); \\
\underline{subcase iii):} we assume that $f(p)^k = p^{it}$ for \emph{all} $p$, and $g(p)^l = p^{it'}$ \emph{only} for large $p$ (see the discussion following the proof of Lemma \ref{SKETCH2}). 
\begin{lem} \label{SKETCH}
Suppose there are positive integers $N$ and $k$ and a real $t$ such that for all $p > N$, $f(p)^k = p^{it}$. Suppose moreover that $g$ is eventually rational. Then $\{(f(n),g(n+1))\}_n$ is dense in $\mb{T}^2$.
\end{lem}
\begin{proof}
By choosing $N$ larger if necessary, and replacing $k$ by some multiple of $k$, we may assume that $g(p)^k = 1$ for all $p > N$ as well. Now suppose $(f(n-1),g(n)) \notin B_{\e}(z,w)$ for $n$ sufficiently large. By Proposition \ref{REDUCE}, $f$ and $g$ are both pseudo-pretentious, and following the argument in the previous section, one can show that $f$ is $h_1n^{it/k}$-pretentious, with $h_1$ a pseudo-character with conductor $q_1$, and similarly $g$ is $h_2$-pretentious\footnote{An application of Lemma \ref{EQUIV}, as in the previous section, implies that $g\bar{h}_2$, which takes finite roots of unity as values, is 1-pretentious.} where $h_2$ is a pseudocharacter with conductor $q_2$. We may assume that $t \neq 0$ here, otherwise $f$ and $g$ are both eventually rational, which is a case dealt with by Proposition \ref{EVRAT}, proven in the previous section.\\
Since $\{g(n)\}_n$ is dense, but with rational argument for all but finitely many primes, there exists a prime $p$ be such that $g(p) = e(\alpha)$ with $\alpha \notin \mb{Q}$. We now fix an even integer $B=2q_1q_2p^r$, where the exponent $r$ is chosen (via Kronecker's theorem) such that $g(p)^r = \bar{g(2q_1q_2)}\bar{w} + O(\e^2)$. Now, if $m$ is chosen so that $P^-(m(Bm-1)) > N$, it follows that 
$$f(Bm-1) = f_0(Bm-1)(Bm-1)^{it/k} = f_0(Bm-1)B^{it/k} m^{it/k} + O(\e^2)$$
for $m$ sufficiently large. Thus, it follows that 
\begin{equation}\label{BADDF0G0}
(f_0(Bm-1)m^{it/k},g_0(m)) \notin B_{\e/2}(zB^{-it/k},1)
\end{equation}
for all $m$ sufficiently large with $P^-(m(Bm-1)) > N$.\\
Now, by Lemma\footnote{Strictly speaking, Lemmata \ref{GETHS} and \ref{ONES} deal only with the pair of linear forms $(n\mapsto n,n\mapsto Bn+1)$; however, the same results can be derived with $(m\mapsto Bm-1,m\mapsto m)$ with minimal change to the proofs.} \ref{GETHS}, for all but $O(\e^3\Phi_{N,B}(x))$ integers $m \leq x$ (with $x$ chosen from appropriate infinite increasing sequence and $N$ suitably large) with $P^-(m(Bm-1)) > N$ we have 
$$(f_0(Bm-1),g_0(m)) = (h_1(Bm-1),h_2(m)).$$ 
Moreover, by Lemma \ref{ONES} there are $\gg \e^2 \Phi_{N,B}(x)$ such integers for which $(h_1(Bm-1),h_2(m)) = (1,1)$, and such that $m^{it/k} = zB^{-it/k} + O(\e^2)$. Hence, it follows that 
$$(f_0(Bm-1)m^{it/k},g_0(m)) = (h_1(Bm-1)m^{it/k},h_2(m)) = (zB^{-it/k},1) + O(\e^2),$$
for a positive upper density set of $m$. But this contradicts \eqref{BADDF0G0}. The contradiction implies
$$\overline{\{f(n-1),g(n)\}_n} = \overline{\{f(n),g(n+1)\}_n} = \mb{T}^2,$$
as claimed.
\end{proof}
\begin{lem} \label{SKETCH2}
Suppose there are positive integers $N,k,l$ and real numbers $t,t'$ such that for all $p > N$, $f(p)^k = p^{it}$ and $g(p)^l = p^{it'}$, but for any $m \in \mb{N}$ and any $u \in \mb{R}$ there is an integer $n$ (necessarily prime) such that $f(n)^m \neq n^{iu}$. Then $\{(f(n),g(n+1))\}$ is dense.
\end{lem}
\begin{proof}
If $t/t' \notin \mb{Q}$ then the proof is the same as the corresponding case in the proof of Proposition \ref{RATRATIO} (which only used data about integers $n$ such that $n(Bn+1) > N$ for suitable $B$). \\
Conversely, suppose $t = at'/b$ for some $a,b\in \mb{Z}$, with $b \neq 0$, and suppose there is an $\e > 0$ and a pair $(z,w) \in \mb{T}^2$ such that $(f(n),g(n+1)) \notin B_{\e}(z,w)$ for all $n$ large. We may assume that $tt' \neq 0$, since otherwise $f$ and $g$ are eventually rational, and this was covered in the previous section. Furthermore, since we may replace $k$ by $kb$ and $l$ by $la$, we may assume that $t = t'$. \\
By hypothesis, we can choose a prime $p_0 \leq N$ for which $f(p_0)^k/p_0^{it}$ is not a root of unity of any order. Indeed, if every $p \leq N$ satisfied $f(p)^k/p^{it} \in \mu_m$ for some $m \geq 1$, we could replace $k$ with $km$ and $t$ with $tm$ so that $f(p)^{km} = p^{itm}$. Making these replacements for $k$ and for $t$ iteratively at every prime $p \leq N$ would imply that $f(n)^{k'} = n^{iu}$ for \emph{all} $n$, where $k' \in \mb{N}$ and $u \in \mb{R}$. This contradicts our initial hypothesis regarding $f$. \\
Thus, as per the above paragraph we can choose $p_0 \leq N$ such that $f(2p_0)(2p_0)^{-it} = e(\alpha)$, where $\alpha \notin \mb{Q}$. We may thus choose $\ell$ such that 
$$(f(2p_0)(2p_0)^{-it})^{\ell} = e(\alpha \ell) = z\bar{w} + O(\e^2).$$ 
Now,\footnote{This can be done, for example, by applying Proposition \ref{ONES} with every pair of values of $h_1(n)$ and $h_2(Bn+1)$, then combining all of these contributions.} we pick $m$ with $P^-(m((2p_0)^{\ell}m+1)) > N$ and such that $m^{it} = w + O(\e^2)$. By assumption, it follows that 
$$g((2p_0)^{\ell}m+1) = ((2p_0)^{\ell}m+1)^{it} = (2p_0)^{i\ell t}m^{it} + O(\e^2),$$
for sufficiently large $m$. Then, setting $n = (2p_0)^{\ell}m$ and assuming that $m$ is sufficiently large, we have
\begin{align*}
(f(n),g(n+1)) &= (f(2p_0)^{\ell} m^{it},(n+1)^{it}) = n^{it}((f(2p_0)(2p_0)^{-it})^{\ell},1) + O(\e^2) \\
&= w \cdot (z\bar{w},1) + O(\e^2) = (z,w) + O(\e^2),
\end{align*}
in contradiction to the claim.
\end{proof}
Turning to subcase iii), it remains to consider the case that $f(n)^k = n^{it}$ for some $k \in \mb{N}$ and $t \in \mb{R}$, and $g(p)^l = p^{it'}$ for all $p > N$ but such that for each $m \in \mb{N}$ and $u \in \mb{R}$ there is $n \in \mb{N}$ with $g(n)^m \neq n^{iu}$. The argument is symmetric to that given in subcase ii) (perhaps up to considering the sequence $\{f(Bn-1),g(n)\}_n$, for a suitable choice of $B$). 
We leave the details to the interested reader.
\section{Proof of Theorem \ref{IRRAT}: The Remaining Cases} \label{SECIRRAT}
In this section, we prove Theorem \ref{IRRAT}, except in the exceptional cases that were dealt with in the previous section. We will assume here that for all $k \in \mb{N}$ and $t \in \mb{R}$ there are infinitely many primes $p$ such that $f(p)^k \neq p^{it}$; by symmetry we may assume that the same holds for the function $g$. \\
First, suppose for contradiction that there is an $\e > 0$ and a point $(z,w) \in \mb{T}^2$ such that 
$$(f(n),g(n+1)) \notin B_{\e}((z,w)) \text{ for all sufficiently large $n$.}$$ 
Applying Proposition \ref{REDUCE}, we know that $f$ and $g$ are, respectively, $h_1n^{it}$- and $h_2n^{it'}$-pseudo-pretentious, where the conductors of the pseudocharacters $h_1$ and $h_2$ are $q_1$ and $q_2$, respectively. \\
We need two technical results in order to proceed in the proof of Theorem \ref{IRRAT}. The first
will allow us to simultaneously control the angular distribution of the values of the irrational function $f$ at prime powers $p^m$, as well as the angular distribution of $p^{imt}$, for $t \in \mb{R}$. 
\begin{lem} \label{PPOWER}
Let $f: \mb{N} \ra \mb{T}$ be completely multiplicative function such that for all $t \in \mb{R}$ and $l \in \mb{N}$ there are infinitely many primes $p$ such that $f(p)^l \neq p^{it}$. 
Let $z \in \mb{T}$, $u \in \mb{R}$ and $\delta,\eta \in (0,1)$. Let $I \subset \mb{T}$ be the symmetric arc about 1 with length $2\delta$. Then for any $k,N \in \mb{N}$ there exists $n\in\mb{N}$ with $P^-(n) >N$ and $m\in\mb{N}$ such that the following holds:
\begin{enumerate}[(i)]
\item $|f(n)^m - z| < \eta$,\\

\item $n^{ium} \in I$,\\

\item $k|m$.
\end{enumerate}
\end{lem}
\begin{proof}
We consider two cases. First, suppose there is a prime $p > N$ for which $f(p)$ and $p^{iu}$ are such that if $f(p) = e(\alpha)$ and $p^{iu} = e(\beta)$ then $\{1,\alpha,\beta\}$ are $\mb{Q}$-linearly independent. Equivalently, $\{1,k\alpha,k\beta\}$ is $\mb{Q}$-linearly independent.  
In this case, we can apply Kronecker's theorem to find $m_0$ such that 
$$\|(f(p)^{km_0},p^{iukm_0}) -(z,1)\|_{\ell_1} \leq \min\{\eta,\delta\}^2,$$ 
and the claim follows with $m = km_0$ and $n = p$. \\
Now, suppose that for all primes $p>N$ we have that $f(p)$ and $p^{iu}$ have $\mb{Q}$-linearly dependent arguments. Equivalently, we can find a root of unity $\xi_p$ such that $f(p) = \xi_pp^{iu}$ for each prime $p>N$. We consider two subcases of this case. \\
First, if $\{\xi_p\}_{p > N}$ is a set of roots of unity of \emph{bounded} order, say $M$, then it follows that $f(p)^{M!} = p^{iuM!}$ for all $p > N$. This contradicts our initial assumption that $f(p)^k \neq p^{it}$ for infinitely many $p$ (with $k = M!$ and $t = uM!$). \\
Next, suppose that $\{\xi_p\}_{p>N}$ is such that for any $M \geq 1$ we can find a prime $p > N$ such that $\xi_p = e(a/b)$ with $b > Mk$ and $(a,b) = 1$. Pick $M > 2\eta^{-2}$, and choose $p = p(M)$ in this way. For $g,r \in \mb{N}$ parameters to be chosen, note that
$$f(p)^{k(gb + r)} = e(ark/b) p^{ik(gb+r)u}.$$
Writing $z= e(\gamma)$, we may select $ 0 \leq \ell \leq b-1$ such that $\gamma \in [k\ell/b,k(\ell + 1)/b]$, whence
$$|\gamma - k\ell/b| \leq k/b \leq 1/M < \frac{1}{2}\eta^2.$$ 
We will thus pick $r$ such that $ar \equiv \ell (b)$, so that $f(p)^{k(gb+r)} = z p^{ik(gb+r)u}.$ \\
Now, if $u = 0$ then we are done, as we can take $n = p$ and $m = kr$ (i.e., with $g = 0$). Otherwise, the sequence $\{p^{ikbug}\}_g$ is dense in $\mb{T}$. Thus, having already chosen $r$, it follows that we can pick $g$ such that $|p^{ikgbu} -p^{-ikru}| \leq \frac{1}{2}\min\{\eta^2,\delta^2\}$. Hence, we get that 
\begin{align*}
|p^{ik(gb+r)u} - 1| &= |p^{ikgbu} -p^{ikru}| \leq \delta^2,\\
|f(p)^{k(gb+r)} -z| &\leq |f(p)^{k(gb+r)} - z p^{ikgbu}| + |z||p^{ikgbu}-1| \leq \eta^2.
\end{align*}
Thus, selecting $m = k(gb+r)$ and $n = p$ suffices to prove the claim in this case.
\end{proof}
\begin{lem} \label{NOTZERO}
Let $f,g: \mb{N} \ra \mb{T}$ be completely multiplicative functions such that $\bar{\{f(n)\}_n} = \bar{\{g(n)\}_n} = \mb{T}$, but that $\bar{\{(f(n),g(n+1))\}_n} \neq \mb{T}^2$, and furthermore that $f$ satisfies the hypotheses of Lemma \ref{PPOWER}. Suppose that $f$ and $g$ are, respectively, $h_1n^{it}$- and $h_2n^{it'}$-pretentious, where $h_1$ and $h_2$ are pseudocharacters of conductor $q_1$ and $q_2$, respectively, and $h_1^{kr} = h_2^{ls} = 1$.
Finally, let $N,B \geq 1$, with $2q_1q_2|B$, and let $u \in \mb{R}$ and $z \in \mb{T}$. Then for any $\eta > 0$ sufficiently small (in terms of $u$) 
there is a 
positive upper density subset $T_z = T_z(\eta)$ of $\mc{A}_{N,B,I}(h_1,h_2;1,1;u)$ on which $|f(n)-z| < \eta$, where $I$ is the symmetric arc of length $2\eta$ about 1.
\end{lem}
In words, the lemma states that we can find a ``large'' set of integers $n$ satisfying $P^-(n(Bn+1)) > N$ such that $f(n)$ is close to $z$, $h_1(n) = h_2(n) = 1$, and $n^{iu} \in I$. 
\begin{proof}
By Lemma \ref{PPOWER}, choose $n_0$ with $P^-(n_0) > N$ and $m$ a multiple of $kr$ such that $f(n_0)^m = z + O(\eta^2)$ and $(n_0)^{imu} = 1 + O(\eta^2)$, and set $B' := Bn_0^m$. Since $f$ is pretentious to $h_1(n)n^{it}$, by Lemma \ref{CHEB} we can choose $x$ from a suitable infinite sequence (and $N$ slightly larger if necessary) such that for all but $O(\eta^3/(klrsn_0^m)\Phi_{N,B'}(x))$ integers $n \leq x$ with $P^-(n(B'n+1)) > N$ we have $f(n) = h_1(n)n^{it} + O(\eta^2)$. Now, furthermore, Proposition \ref{ONES} gives
$$|\mc{A}_{N,B',I,I}(h_1,h_2;1,1;t,u) \cap [1,x/n_0^m]| \gg \frac{\eta^2}{klrs}\Phi_{N,B'}(x/n_0^m) \asymp \frac{\eta^2}{klrsn_0^m}\Phi_{N,B'}(x),$$
provided that $t/u \notin \mb{Q}$. If, instead, $t/u = a/b$ and $a < b$ and $\eta^{-1/2} > b/a$ (without loss of generality), we can conclude that if $n^{it} = 1 + O(\eta^2b/a)$ then $n^{iu} = 1 + O(\eta^{3/2})$, and hence
$$|\mc{A}_{N,B',I,I}(h_1,h_2;1,1;t,u) \cap [1,x/n_0^m]| \geq |\mc{A}_{N,B,I'}(h_1,h_2;1,1;t) \cap [1,x/n_0^m]| \gg \frac{\eta^2}{klrsn_0^m}\Phi_{N,B'}(x),$$
where $I'\subset I$ is a symmetric subinterval of length $\eta^2$ about 1. In either case, we see that on a positive upper density subset of $\mc{A}_{N,B',I,I}(h_1,h_2;1,1;t,u)$, we have $$f(n) = h_1(n)n^{it} + O(\eta^2) = 1 + O(\eta^2).$$
Now, let $n \in \mc{A}_{N,B',I,I}(h_1,h_2;1,1;t,u)$, and set $n' := n_0^mn$. Then $h_1(n') = h_1(n_0)^mh_1(n) = 1$, since $kr|m$ and $h^{kr} = 1$. Furthermore, we have $h_2(Bn'+1)= h_2(B'n+1) = 1$ by assumption. Moreover, we get $(n')^{iu} = (n_0^m)^{iu} n^{iu} = 1 + O(\eta^2)$ by our assumptions on $n_0$ and $n$, and finally $f(n')= f(n_0)^mf(n) = z + O(\eta^2)$. It follows upon setting
$$T_z := \{n' = n_0^mn : n \in \mc{A}_{N,B',I,I}(h_1,h_2;1,1;t,u)\}$$
that $T_z$ has positive upper density (the same as that of $\mc{A}_{N,B',I,I}(h_1,h_2;1,1;t,u)$), that $T_z \subseteq \mc{A}_{N,B,I}(h_1,h_2;1,1;u),$ and that on $T_z$ we have $f(n') = z + O(\eta^2)$. This implies the claim if $\eta$ is small enough.
\end{proof}
\begin{proof}[Proof of Theorem \ref{IRRAT}, Part 2] 
We assume here that $f$ is both irrational, and such that for each $k \in \mb{N}$ and $t \in \mb{R}$, there are infinitely many primes $p$ for which $f(p)^k \neq p^{it}$; we can do this since the other cases where $f$ is irrational were treated in the previous section. \\
Now, suppose that $\{(f(n),g(n+1)\}_n$ stays outside of an arc of radius $\e$ (in $\ell^1$) about $(a,b) \in \mb{T}^2$. This means that $|f(n)-a| + |g(n+1)-b| \geq \e$ for all large $n$. 
By Proposition \ref{REDUCE}, there are minimal positive integers $k,l,r$ and $s$, 
pseudocharacters $h_1,h_2$ with conductors $q_1,q_2 = O_{\e}(1)$ taking values in $\mu_{kr}$ and $\mu_{ls}$ respectively, and $t_1,t_2\in \mb{R}$ such that $\mb{D}(f,h_1n^{it_1};x), \mb{D}(g,h_2n^{it_2};x) \ll_{\e} 1$. \\
Let $I$ be a symmetric interval of length $\e^2$ about 1, and let $B$ be an integer of our choosing, chosen subject only to the constraint that $2q_1q_2|B$. Consider those $n \in \mc{A}_{N,B,I}(h_1,h_2;1,1;t_1-t_2) =: T$. By Proposition \ref{ONES}, $T$ has positive upper density. We have three possible scenarios:
\begin{enumerate}[(i)]
\item there is a positive upper density subset of $T$ on which $|f(Bn)-a| \geq \frac{3\e}{4}$ and $|g(Bn+1)-b| < \frac{\e}{4}$; 
\item there is a positive upper density subset of $T$ on which $|g(Bn+1)-b| \geq \frac{3\e}{4}$ and $|f(Bn)-a| < \frac{\e}{4}$;
\item except on an upper density zero subset of $T$ we have $|f(Bn)-a|\geq \frac{3\e}{4}$ and $|g(Bn+1)-b| \geq \frac{\e}{4}$.
\end{enumerate}
Consider alternative (iii). 
It is clearly true that $|f(n)-a\bar{f(B)}| \geq \frac{3\e}{4}$ for all but a zero upper density subset of $\mb{N}$. By Lemma \ref{NOTZERO} (taking $u = t_1-t_2$ and $z = a\bar{f(B)}$ in the notation there), we can find a positive upper density subset $T_{a\bar{f(B)}} \subset T$ on which $|f(n)- a\bar{f(B)}| < \frac{3\e}{4}$, contradicting the conditions of case (iii). Thus, (iii) can clearly not occur. \\
Suppose next that we are in case (i). 
Then
\begin{equation*}
|f(Bn)\bar{g(Bn+1)}-a\bar{b}| \geq |f(Bn)-a| - |g(Bn+1)-b| \geq \frac{\e}{2},
\end{equation*}
on a subset of positive upper density in $T$. The same condition occurs in case (ii) as well (by essentially the same argument).\\
Now, put $F(n) := f(n)\bar{h_1(n)} n^{-it_1}$ and $G(n) := g(n)\bar{h_2(n)}n^{-it_2}$. We recall that for $n \in T$, $h_1(n) = h_2(n)= 1$ and $n^{i(t_2-t_1)} = 1 + O(\e^2)$. For each $n \in T$ sufficiently large we have
\begin{align*}
F(n)\bar{G(Bn+1)} &= f(n)\bar{g(Bn+1)} n^{-it_1}(Bn+1)^{it_2} = f(Bn)\bar{g(Bn+1)} \cdot n^{i(t_2-t_1)} \bar{f(B)}B^{it_2}\\
&= f(Bn)\bar{g(Bn+1)} \cdot \bar{f(B)}B^{it_2} + O(\e^2).
\end{align*}
It follows that on some positive upper density subset of $T$, we have
\begin{equation*}
|F(n)\bar{G(Bn+1)} - \bar{b}a\bar{f(B)}B^{it_2}| \geq \frac{\e}{3},
\end{equation*}
for $\e$ sufficiently small. At this point, we will select $B$ using the following remarks. First, we are supposing that for any $k \in \mb{N}$ and $t \in \mb{R}$ there is a prime $p$ such that $f(p)^k \neq p^{it}$. As such, for any $M \geq 1$ we can find $P$ sufficiently large in terms of $M$ such that $f(P)P^{-it_2} = e(\alpha)$ where either $\alpha \notin \mb{Q}$ or else $\alpha \in \mb{Q}$ with denominator at least $M$. Taking $M \asymp \e^{-2}$, we can choose $r$ such that $\alpha r$ lies in the same $O(\e^2)$-length arc of $\mb{T}$ as the argument of $b\bar{af(2q_1q_2)}$. Consequently, with this choice of $r$ we find that
$$f(P)^rP^{-irt_2} = b(2q_1q_2)^{it_2}\bar{af(2q_1q_2)} + O(\e^2).$$ 
Finally, we choose $B = 2q_1q_2P^r$, and with this choice 
$$f(B)B^{-it_2} = \bar{b}a + O(\e^2).$$ 
It then follows that
\begin{equation*}
|F(n)\bar{G(Bn+1)} - 1| \geq \frac{\e}{4}
\end{equation*}
for each $n$ belonging to a positive upper density subset of $T$. \\
Now, with $N$ large we select by Szemer\'{e}di's theorem (see \cite{Gowers}) a long arithmetic progression $S \subset T$; in particular, $|S| \ra \infty$ as $x \ra \infty$. As in the proof of Theorem 2.1 in \cite{RIG}, we have that on one hand
\begin{align*}
\frac{1}{16}\e^2|S\cap [1,x]| &\leq \sum_{n \leq x \atop n \in S} |F(n)\bar{G}(Bn+1) - 1| \\
&= 2|S\cap [1,x]|\left(1-\text{Re}\left(\frac{1}{|S\cap [1,x]|} \sum_{n \leq x \atop n \in S} F(n)\bar{G(Bn+1)}\right)\right).
\end{align*}
On the other hand, since $F$ and $G$ are both 1-pretentious and $N$ is large enough we have by Lemma 2.13 of \cite{RIG} that
\begin{equation*}
\sum_{n \leq x \atop n \in S} F(n)\bar{G(Bn+1)} = |S \cap [1,x]|\left(1 + O(\e^3)\right),
\end{equation*}
as $x \ra \infty$ along a suitable infinite subsequence, and the implicit constant is independent of $\e$. This yields the inequality
\begin{equation*}
\frac{1}{16}\e^2 |S\cap [1,x]| \ll \e^3|S\cap [1,x]|,
\end{equation*}
which is patently false when $\e$ is sufficiently small. \\
Note that we can argue in precisely the same way with the roles of $f$ and $g$ reversed to conclude that case (ii) cannot occur provided that $g(n)^l \neq n^{it'}$ for some $t'$ (looking at the forms $(g(n),f(n-1))$ instead). This contradiction completes the proof of Theorem \ref{IRRAT} in those cases not covered in the previous section.
\end{proof}

\end{document}